\documentclass[11pt]{article}
\usepackage[margin=1in]{geometry}
\usepackage{algorithm}
\usepackage{algpseudocode}
\usepackage{url} 

\providecommand{\headers}[2]{}

\let\oldtitle\title
\renewcommand{\title}[2][]{\oldtitle{#2}}

\makeatletter
\newcommand{\LoadPackageUnlessAcm}[2][]{%
    \@ifclassloaded{acmart}{%
        \PackageInfo{preamble}{Skipping package `#2' because acmart class is used}%
    }{%
        \if\relax\detokenize{#1}\relax
            \usepackage{#2}%
        \else
            \usepackage[#1]{#2}%
        \fi
    }%
}
\makeatother

\LoadPackageUnlessAcm{caption}
\LoadPackageUnlessAcm{subcaption}
\LoadPackageUnlessAcm{authblk}
\LoadPackageUnlessAcm{hyperref}

\usepackage{amstext}
\usepackage{stmaryrd}
\usepackage{algpseudocode}

\usepackage{units}
\usepackage{cancel}
\usepackage{graphicx}

\usepackage{xcolor}
\usepackage{array}
\usepackage{verbatim}
\usepackage{booktabs}
\usepackage{makecell}
\usepackage{calc}
\usepackage{textcomp}
\usepackage{multirow}
\usepackage{tikz,tikzscale}

\usepackage{pgfplots}
\usepackage{pgfplotstable}
\usepackage{pgfkeys}
\pgfplotsset{compat=1.18}

\definecolor{deepgreen}{RGB}{0,100,0}

\newif\ifrunscripts
\runscriptstrue
\pgfplotsset{
    p2/.style={teal, mark=*},
    p3/.style={orange, mark=triangle*},
    p7/.style={blue, mark=square*},
}

\newcommand{\plotwidth}{0.49\columnwidth}
\newcommand{\plotheight}{0.4\columnwidth}

\pgfplotscreateplotcyclelist{paper markers}{
    {teal,    solid, mark=*},
    {orange,   solid, mark=triangle*},
    {blue,     solid, mark=square*},
    {red,      solid, mark=pentagon*},
    {violet,     solid, mark=otimes*},
    {brown,    solid, mark=halfcircle*},
    {black,   solid, mark=diamond*},
}

\colorlet{BarColorOne}{orange!60!white}
\colorlet{BarColorTwo}{teal!80!white}
\colorlet{BarColorThree}{red!80!black}
\colorlet{BarColorFour}{blue!60!white}
\colorlet{BarColorFive}{brown!70!black}

\pgfplotsset{legend swatch/.style={area legend, draw=none}}

\pgfplotsset{
    paperplot/.style={
            width=\plotwidth,
            height=\plotheight,
            cycle list name=paper markers,
            every axis/.append style={font=\small},
            grid=major,
            thick
        }
}

\pgfplotsset{
    longplot/.style={
            width=0.8\columnwidth,
            height=\plotheight,
            cycle list name=paper markers,
            every axis/.append style={font=\small},
            grid=major,
            thick
        }
}

\usepackage[all]{xy}
\usepackage{mathtools}
\usepackage{placeins}
\usetikzlibrary{calc}
\usetikzlibrary{arrows.meta}
\usetikzlibrary{shadings}
\usepgfplotslibrary{fillbetween}

\usepackage[mode=multiuser,status=draft]{fixme}
\fxsetup{innerlayout=inline}
\fxsetup{targetlayout=colorcb}
\fxusetheme{color}
\FXRegisterAuthor{mw}{envmw}{MW}

\definecolor{apply}{rgb}{0.3,.7,0.}
\definecolor{applyface}{rgb}{.9,.95,0.}
\definecolor{invert}{rgb} {0.5,0.,1.}
\colorlet{invertface}{invert!60!red}

\usepackage[colorinlistoftodos,prependcaption,textsize=small]{todonotes}

\usepackage{colortbl}

\usepackage{booktabs}
\usepackage[font=footnotesize,position=b]{subcaption}
\usepackage{adjustbox}

\makeatletter
\def\pgfplots@getautoplotspec into#1{%
        \begingroup
        \let#1=\pgfutil@empty
        \pgfkeysgetvalue{/pgfplots/cycle multi list/@dim}\pgfplots@cycle@dim
        \let\pgfplots@listindex=\pgfplots@numplots
        \pgfkeysgetvalue{/pgfplots/cycle list set}\pgfplots@listindex@set
        \ifx\pgfplots@listindex@set\pgfutil@empty
        \else
            \c@pgf@counta=\pgfplots@listindex
            \c@pgf@countb=\pgfplots@listindex@set
            \advance\c@pgf@countb by -\c@pgf@counta
            \globaldefs=1\relax
            \edef\setshift{%
                \noexpand\pgfkeys{
                    /pgfplots/cycle list shift=\the\c@pgf@countb,
                    /pgfplots/cycle list set=
                }
            }%
            \setshift%
            \globaldefs=0\relax
        \fi
        \pgfkeysgetvalue{/pgfplots/cycle list shift}\pgfplots@listindex@shift
        \ifx\pgfplots@listindex@shift\pgfutil@empty
        \else
            \c@pgf@counta=\pgfplots@listindex\relax
            \advance\c@pgf@counta by\pgfplots@listindex@shift\relax
            \ifnum\c@pgf@counta<0
                \c@pgf@counta=-\c@pgf@counta
            \fi
            \edef\pgfplots@listindex{\the\c@pgf@counta}%
        \fi
        \ifnum\pgfplots@cycle@dim>0
            \c@pgf@counta=\pgfplots@cycle@dim\relax
            \c@pgf@countb=\pgfplots@listindex\relax
            \advance\c@pgf@counta by-1
            \pgfplotsloop{%
                \ifnum\c@pgf@counta<0
                    \pgfplotsloopcontinuefalse
                \else
                    \pgfplotsloopcontinuetrue
                \fi
            }{%
                \pgfkeysgetvalue{/pgfplots/cycle multi list/@N\the\c@pgf@counta}\pgfplots@cycle@N
                \pgfplotsmathmodint{\c@pgf@countb}{\pgfplots@cycle@N}%
                \divide\c@pgf@countb by \pgfplots@cycle@N\relax
                \expandafter\pgfplots@getautoplotspec@
                \csname pgfp@cyclist@/pgfplots/cycle multi list/@list\the\c@pgf@counta @\endcsname
                {\pgfplots@cycle@N}%
                {\pgfmathresult}%
                \t@pgfplots@toka=\expandafter{#1,}%
                \t@pgfplots@tokb=\expandafter{\pgfplotsretval}%
                \edef#1{\the\t@pgfplots@toka\the\t@pgfplots@tokb}%
                \advance\c@pgf@counta by-1
            }%
        \else
            \pgfplotslistsize\autoplotspeclist\to\c@pgf@countd

            \pgfplots@getautoplotspec@{\autoplotspeclist}{\c@pgf@countd}{\pgfplots@listindex}%
            \let#1=\pgfplotsretval
        \fi
        \pgfmath@smuggleone#1%
        \endgroup
    }

\pgfplotsset{
    cycle list set/.initial=
}
\makeatother

\usepackage{etoolbox}
\AtBeginEnvironment{tabular}{\footnotesize}
\AtBeginEnvironment{tabular*}{\footnotesize}
\AtBeginEnvironment{tabularx}{\footnotesize}

\usepackage{lipsum}
\usepackage{amsfonts}
\usepackage{graphicx}
\usepackage{epstopdf}
\usepackage{amsmath,amssymb,amsthm}
\newtheorem{lemma}{Lemma}
\newtheorem{theorem}{Theorem}

\theoremstyle{definition}

\theoremstyle{remark}

\theoremstyle{plain}
\ifpdf
    \DeclareGraphicsExtensions{.eps,.pdf,.png,.jpg}
\else
    \DeclareGraphicsExtensions{.eps}
\fi

\newcommand{\secref}[1]{Section~\ref{#1}}

\headers{Matrix-Free Finite Elements in Cell-Wise Storage}{Wichrowski}

\title{Coalesced Matrix-Free Finite Elements in Cell-Wise Storage}
\author{Michał Wichrowski}

\usepackage{amsopn}

\providecommand{\keywordsname}{Keywords}
\newcommand{\Keywords}[1]{\par\noindent\textbf{\keywordsname:} #1}

\title{Adaptive Multilevel Discontinuous Galerkin Methods on GPUs}

\begin{document}
\author{Michał Wichrowski$^0$}
\footnotetext{
    % Interdisziplinäres Zentrum für Wissenschaftliches Rechnen (IWR), Ruprecht-Karls-Universität Heidelberg, Germany,
    \texttt{mwichro@mimuw.edu.pl}}
\date{}

\maketitle

\begin{abstract}
    %!TEX root = ../main_dg_shadow.tex

% Summarize the shadow-cell method and its properties.
I present a matrix-free symmetric interior penalty discontinuous Galerkin method for adaptively refined Cartesian
meshes on GPUs.
At non-matching interfaces, auxiliary \emph{shadow cells} represent the adjacent coarse polynomial on the fine level.
This approach converts non-matching interfaces into matching faces, permitting uniform face evaluation throughout the
mesh. I prove that this construction is equivalent to the standard non-matching formulation. Furthermore, this
representation yields a local geometric multigrid method in which frozen shadows provide inter-level boundary data and
carry residual contributions to coarser levels. A primal--dual pairing eliminates shadow assembly from the Krylov
iteration. Numerical experiments with cubic elements ($p=3$) on an NVIDIA A100 show stable multigrid convergence under
increasing refinement depth and efficient GPU execution.

\end{abstract}

\Keywords{geometric multigrid, adaptive mesh refinement, hanging nodes, matrix-free, finite elements, GPU}

% \AMS{65N55, 65N30, 65F08, 65Y10, 65Y20}

%%%%%%%%%%%%%%%%%%%%%%%%
% Introduction to the DG method.
%%%%%%%%%%%%%%%%%%%%%%%%
\section{Introduction}
\label{sec:introduction}
%% TODO: write the introduction. Selling points to lead with:
%%  1. hanging-node constraints no longer needed because they are embedded in the MG structure;
%%  2. novel realization of edge operators (dual-side masking, no A^SE/A^ES blocks);
%%  3. DSS works unchanged on adaptively refined meshes;
%%  4. novel transfers: valence weighting removed entirely (restrict the unassembled dual vector).
%% Headline number: at p=3, masked point-Jacobi MG matches a patch-smoother solver in solve time (fp64 vs fp64).

%!TEX root = ../main_dg_shadow.tex

%%%%%%%%%%%%%%%%%%%%%%%%
% Motivation: matrix-free DG on GPUs, and where the cost actually sits
%%%%%%%%%%%%%%%%%%%%%%%%
% Paragraph: DG and matrix-free evaluation match the data-movement constraints of GPUs.
Discontinuous Galerkin (DG) methods combine high-order accuracy with a discrete structure well suited to GPUs: all
degrees of freedom are element-local, the mass matrix is block-diagonal, and the operator action decomposes into dense,
sum-factorized element kernels~\cite{Kronbichler2017a,bastian2019matrix,fehn2019matrix}. Matrix-free evaluation of DG
operators is correspondingly mature, and sum-factorized kernels can approach the available device bandwidth when their
memory accesses are regular~\cite{Kronbichler2017a,kronbichler2019multigrid,fischer2020scalability}. Because the
floating-point throughput of processors, CPUs and accelerators alike, has
grown faster than their memory bandwidth~\cite{williams2009roofline}, performance is governed primarily by the number
of bytes moved and the regularity with which they are moved.

% Paragraph: adaptive local multigrid concentrates irregular work at refinement edges.
Local multigrid exploits the refinement hierarchy by decomposing an adaptive mesh into levels and smoothing only the
locally refined region on each level; its work therefore follows the adaptive hierarchy rather than a sequence of
globally refined meshes~\cite{brandt1977mlat,mccormick1986fac,JanssenKanschat11}. The operator on each level must
nevertheless account for coupling across its refinement edge. Evaluating the level residual therefore requires
integrating the faces on this edge: level prolongation supplies the coarse-side traces, while level restriction returns the
resulting coarse-side residual contributions to the coarser level. Classical local-smoothing formulations expose this
coupling through separate interior and edge operators and constrained inter-grid transfers~\cite{JanssenKanschat11}.

% Paragraph: state of the art in DG multigrid, GPU realizations, and current treatments of adaptivity.
Fast solvers for interior penalty discretizations are by now well developed. Multilevel methods for DG were introduced
and analyzed in~\cite{gopalakrishnan2003multilevel,brenner2005convergence} and extended to $hp$ hierarchies
in~\cite{antonietti2015multigrid}; with block smoothers that respect the element-local structure of the broken
space, such as block Jacobi~\cite{pazner2018approximate} and block Gauss--Seidel~\cite{Kanschat08smoother},
overlapping Schwarz, and patch
smoothers~\cite{WitteArndtKanschat21,cui2025implementation,wichrowski2025local}, they yield level-independent
convergence for the symmetric interior penalty operator. Where a geometric hierarchy is unavailable, the coarse spaces
are built by agglomeration~\cite{antonietti2020agglomeration},
algebraically~\cite{prill2009smoothed,bastian2012algebraic}, or by degree coarsening to a low-order or continuous space,
which is the basis of the hybrid multigrid solvers used in
matrix-free frameworks~\cite{kronbichler2018performance,fehn2020hybrid,bastian2019matrix}. On accelerators these
methods are almost always realized matrix-free, since storing and streaming assembled DG matrices would incur an excessive memory footprint and bandwidth requirement: nodal DG
kernels were mapped to GPUs in~\cite{kloeckner2009nodal}, and sum-factorized, vectorized DG operators with multigrid
preconditioners are available in production frameworks~\cite{Kronbichler2017a,muthing2017high,bastian2019matrix,
    fehn2019matrix,ljungkvist2017matrix,kronbichler2019multigrid}. Closest to the present work,
Cui and Kanschat~\cite{cui2025multilevel} study the data and compute layouts of a matrix-free multigrid solver for the
interior penalty operator on GPUs, with fast-diagonalized tensor-product smoothers and mixed precision, on uniformly
refined meshes; we use their reported solver throughput as a reference point in \secref{sec:results}.

% Paragraph: what these approaches do at the refinement edge, and what is therefore still open.
Adaptivity is where the picture becomes less uniform. Global multigrid on a sequence of uniformly refined meshes is the
simplest option, but on a locally refined mesh its levels carry the cost of the globally refined hierarchy rather than
of the refined region. Local smoothing avoids this by restricting each level to its own
cells~\cite{brandt1977mlat,mccormick1986fac,JanssenKanschat11} and is implemented in this form in parallel adaptive
frameworks~\cite{clevenger2021flexible,kronbichler2019multigrid,ljungkvist2017matrix}; the price is a level operator
split into interior and edge parts, with dedicated edge kernels and constrained transfers that realize the inter-level
boundary condition~\cite{JanssenKanschat11}. Independently of the solver, the discretization itself must integrate the
non-matching faces of the adaptive mesh. Matrix-free DG implementations do so with a second family of face kernels
using subface quadrature and per-subface interpolation~\cite{Kronbichler2017a}, while the spectral-element literature
routes the same coupling through mortar spaces on the interface~\cite{kopriva1996conservative,kopriva2002mortar}. Both
remedies introduce, at the refinement edge, precisely the irregular addressing and divergent execution paths that the
matrix-free DG kernels were designed to avoid, and they do so in the operator, the smoother, and the level transfers
alike. The present work eliminates this separate non-matching treatment altogether: the refinement edge is expressed in the same
matching-face kernels as the rest of the mesh, and the local multigrid cycle inherits that uniformity.

%%%%%%%%%%%%%%%%%%%%%%%%
% Our approach
%%%%%%%%%%%%%%%%%%%%%%%%

% Paragraph: contribution 1 replaces the face list with face exchanges.
Our first contribution (\secref{sec:faces}) is the organization of the face integrals as \emph{structured face
exchanges}. Within a structured block, the faces normal to each axis are enumerated arithmetically and processed by one
dimensionally-split sweep per axis: each thread owns one interface, extracts the two coincident face fragments at
statically known offsets,
evaluates the numerical flux, and writes one contribution back to each side. Neighbor access becomes a direct load
instead of an indirect gather, the two-sided write is race-free without atomics or coloring, and indirection survives
only on the skeleton between macro-blocks.

% Paragraph: contribution 2 introduces shadow cells for non-matching faces.
Our second and main contribution is the treatment of adaptive refinement by \emph{shadow
cells} (\secref{sec:shadow}). Instead of integrating over subfaces, we artificially refine every cell on
the coarse side of a refinement edge (under the standard 2:1 balance assumption) and allocate its children in storage as
shadows: cells that are structurally identical to active cells but carry no independent unknowns, their content always
derived from the
parent by polynomial prolongation. Every face of every fine cell then finds a matching partner at its own level, and
the face loop runs uniformly: refresh the shadows, sweep over matching faces including the fine-shadow ones, and
accumulate the shadow contributions onto the parent via transposed prolongation (shadow restriction). We prove that this reproduces the non-matching face
integrals and quadrature of the standard interior penalty method; the construction reorganizes the computation without
introducing an approximation.

% Paragraph: contribution 3 turns shadow redundancy into a saving through primal-dual pairings.
Shadow storage is redundant, and our third contribution (\secref{sec:dg_pairing}) exploits this redundancy within a
primal--dual framework. Following the primal--dual formulation for redundant storage in continuous
elements~\cite{wichrowski2026DSS},
we pair a consistent primal vector with an unassembled dual one, and show that the blockwise pairing over active and
shadow blocks already
equals the assembled active-mesh pairing. Conjugate gradients therefore run on split data: explicit shadow assembly is eliminated
from every operator application within the Krylov iteration, vector updates preserve consistency of shadow values, and
the only inter-level transfers are those inherent to the multigrid cycle. Consequently, solver convergence is monitored
via the duality pairing rather than the standard Euclidean residual norm.

% Paragraph: contribution 4 uses local multigrid with frozen shadows at the refinement edge.
Our fourth contribution (\secref{sec:local_mg}) is a geometric multigrid preconditioner with local
smoothing~\cite{brandt1977mlat,
    mccormick1986fac,JanssenKanschat11}, following the cell-wise matrix-free V-cycle
of~\cite{wichrowski2026MG}, in which the shadow cells double as the inter-level boundary data. During the
smoothing sweeps on a given level, the solution on the unrefined side of the refinement edge does not change because it is
corrected by coarser levels of the V-cycle. We therefore refresh the coarse-side shadow values once per
level visit; subsequent smoothing sweeps and residual evaluations use the uniform face loop with frozen shadow data.
Because shadows are not relaxed, this enforces the inter-level Dirichlet condition of the classical local-smoothing
formulation. The residual evaluation accumulates outward fluxes on the shadows, and shadow assembly routes them to the
coarser level. Thus the refinement edge uses the shadow-embedding, face-exchange, and shadow-assembly kernels.

% Paragraph: computational overhead and scope.
The costs of the approach are explicit and local: a shadow layer of $2^d$ cells per coarse cell along each refinement
edge, giving surface-order storage and work, and one shadow refresh and one shadow assembly per level visit, not per operator
application (\secref{sec:dg_pairing}). In exchange, the solver contains no non-matching
quadrature, no hanging-node bookkeeping, no atomics, and no indirect addressing outside the macro-block skeleton. On the
A100 the face kernel reaches $81\%$ of peak memory bandwidth, and at four levels of local refinement the operator
retains $71.5\%$ of its uniform-mesh throughput while the iteration count stays constant. The detailed analysis of the
kernels, the storage layout, and the behavior under adaptive refinement is presented in \secref{sec:results}.

% Paragraph: the implementation vehicle.
The entire solver is implemented in Triton~\cite{triton2019}, a tile-based GPU language originating in the
machine-learning community, as are the companion solvers of~\cite{wichrowski2026DSS,wichrowski2026MG}. Its programming
model operates on dense, power-of-two tiles and matches the blocked cell-wise layout used here: the compiler handles
coalescing, shared-memory bank conflicts, and the mapping of tile-level contractions onto the FP64 tensor cores, which
would otherwise require hand-tuned CUDA. Because its compiler emits both PTX and its AMD counterpart, the same kernel
source is not tied to one vendor.

% Paragraph: additionally investigate the blocked cell-wise layout as a performance parameter.
Additionally, we investigate the blocked cell-wise storage layout of \secref{sec:dg_cellwise} as a performance
parameter. Such layouts were
previously studied for continuous elements in~\cite{wichrowski2026DSS}. For DG methods, element-local storage is
already the native representation of the broken space, and no global numbering constraint fixes the memory addresses of
the degrees of freedom. This freedom allows the persistent data layout to be decoupled from the mesh topology. We
benchmark the blocked $[B,\,n_x,\,n_y,\,n_z,\,n_e]$ layouts for the interior penalty operator in
\secref{sec:dg_results_face_kernel}.

%!TEX root = ../main_dg_shadow.tex

% DG jump/average notation (local to this paper; move to preamble if reused)
\providecommand{\jump}[1]{[\![#1]\!]}
\providecommand{\avg}[1]{\{\!\!\{#1\}\!\!\}}

%%%%%%%%%%%%%%%%%%%%%%%%
% DG discretization and blocked storage
%%%%%%%%%%%%%%%%%%%%%%%%
\section{Discontinuous Galerkin Discretization}
\label{sec:dg_cellwise}

% Paragraph: setting and model problem.
We consider the Poisson model problem on a bounded domain $\Omega \subset \mathbb{R}^d$, $d \in \{2,3\}$, with
homogeneous Dirichlet boundary conditions, discretized by the symmetric interior penalty discontinuous Galerkin (SIPG)
method~\cite{arnold1982interior,arnold2002unified} on a \emph{Cartesian} mesh $\mathcal{T}$ of axis-aligned
quadrilateral or hexahedral cells with tensor-product polynomials $\mathbb{Q}_p$. The mesh originates from a macro-grid
of axis-aligned blocks whose cells are subdivided into structured sub-blocks; the computational grid is obtained by
hierarchical refinement, either uniform or adaptive. Since refinement by bisection preserves axis alignment, every cell
is a scaled translate of the reference cube, and its geometry reduces to one scaling factor per refinement level. The
formulation also applies to symmetric second-order elliptic operators with cell-wise constant coefficients. The method
and its kernels are formulated for general degree $p$, but the implementation and every measurement reported in
\secref{sec:results} use cubic elements, $p = 3$; statements about the $p$-dependence of the data traffic and of the
optimal block size are analysis rather than measurement.
%% NOTE: p = 3 is the only measured degree. Keep p-dependent claims explicitly labelled as analysis.
The discrete bilinear form reads, in standard notation with jumps $\jump{\cdot}$ and averages $\avg{\cdot}$ on the face
skeleton $\mathcal{F}_h$,
\begin{equation} \label{eq:sipg}
    a_h(u, v) = \sum_{K \in \mathcal{T}} \int_K \nabla u \cdot \nabla v \, dx
    - \sum_{F \in \mathcal{F}_h} \int_F \Bigl( \avg{\partial_n u} \jump{v} + \jump{u} \avg{\partial_n v}
    - \sigma_F \jump{u} \jump{v} \Bigr) \, ds,
\end{equation}
with the usual modifications of jump and average on boundary faces and a penalty parameter $\sigma_F \sim p^2 / h_F$.
The discrete system $A u = \tilde{b}$ is solved matrix-free: the action of $A$ is recomputed on the fly by
sum-factorized evaluation of the volume and face integrals~\cite{Kronbichler2017a,bastian2019matrix}.

% Paragraph: the operator split, stated once and used throughout.
The two sums of~\eqref{eq:sipg} are evaluated by two separate families of kernels, and we keep them apart throughout
the paper. Writing $a_h = a^{\text{vol}} + a^{\text{face}}$ for the volume and skeleton parts of the form, the discrete
operator splits accordingly,
\begin{equation} \label{eq:operator_split}
    A = A^{\text{vol}} + A^{\text{face}},
\end{equation}
with $A^{\text{vol}}$ the sum-factorized volume kernel of \secref{sec:hermite} and $A^{\text{face}}$ the face
sweeps of \secref{sec:faces}. On a uniformly refined mesh the split is formal: both parts run over
every cell of the mesh, and one may just as well fuse them into a single kernel. Its significance appears under
adaptive refinement, where the two parts run over genuinely \emph{different} cell sets: the volume kernel over
the cells carrying degrees of freedom, the face sweeps over a larger set that includes the auxiliary cells of
\secref{sec:shadow}. Separating them is what allows the face loop to be made uniform without duplicating volume
contributions.

% Paragraph: explain the geometry-traffic argument for Cartesian grids and the cutFEM route.
The restriction to Cartesian grids is a deliberate design decision because it controls data movement. A matrix-free
operator on the storage layout of this work is memory-bound over much of the degree range~\cite{wichrowski2026DSS}, so
its cost depends strongly on the bytes streamed per degree of freedom. Reading and writing the field itself costs two
doubles per degree of freedom. On a deformed mesh, the quadrature loop streams per-point geometry data on top of that:
in three dimensions, six entries of the symmetric metric tensor plus the Jacobian determinant. This geometry traffic
can dominate deformed-cell kernels on GPUs and erode the bandwidth advantage of element-local storage. We therefore
retain a Cartesian background mesh. Complex geometries may be treated by unfitted
discretizations~\cite{burman2015cutfem,bergbauer2025high,cui2025multigrid} or shifted-boundary
variants~\cite{wichrowski2025matrix}, which confine geometric irregularity to a lower-dimensional layer while the bulk
retains the structure exploited here; coupling such a layer to the present kernels is beyond our scope. On Cartesian
cells, the kernels stream no geometry data, the face integrals factorize into precomputed one-dimensional matrices
(\secref{sec:faces}), and the trace scalings of the Hermite-type basis (\secref{sec:hermite}) are exact per-level
constants.

% Paragraph: storage layout as a free parameter.
The broken space $\mathbb{V}_{DG}$ carries no inter-element continuity: every degree of freedom belongs to one cell.
Element-local storage is therefore the standard representation in DG codes, and the solution vector is a concatenation
of per-cell blocks. Since no global numbering constraint ties a DoF to any particular memory address, the in-memory
arrangement of the cell blocks, and of the DoFs within and across them, is a free tuning parameter of the method. To
the high-performance DG frameworks cited here, the layout is typically fixed (often as contiguous per-element blocks in
mesh order) and the kernels are optimized around it~\cite{Kronbichler2017a,muthing2017high,bastian2019matrix,
    fehn2019matrix}. Here, we instead treat the layout itself as a tuning parameter.

% Paragraph: the blocked layout.
Concretely, we store all field data in the blocked multi-dimensional array $[B,\, n_x,\, n_y,\, n_z,\, n_e]$ introduced
for continuous elements in~\cite{wichrowski2026DSS}: $B$ spatially adjacent elements form one memory tile, with $(n_x,
    n_y, n_z)$ the local tensor-product DoF indices and $n_e$ enumerating the elements of a block. The two extremes of the
design space pull in opposite directions. An element-major layout gives perfectly coalesced cross-element access for
the volumetric kernels, but scatters the face fragments of a single element across distant addresses; a node-major
layout keeps each element compact, which is ideal for extracting face traces, but forfeits coalescing across elements.
The block size $B$ interpolates between the two, and its optimum depends on the mix of volumetric and face work, which
in DG is governed by the polynomial degree through the surface-to-volume ratio of the data. We benchmark this trade-off
for the volume and face components of the SIPG operator in \secref{sec:dg_results_face_kernel}.

% Paragraph: adaptivity introduces redundancy into DG storage.
On a uniform DG mesh, the residual on the active cells requires no assembly. Adaptive refinement introduces redundant
storage through the shadow cells of \secref{sec:shadow}; \secref{sec:dg_pairing} shows how a primal--dual formulation
eliminates shadow assembly from the Krylov loop.

%%%%%%%%%%%%%%%%%%%%%%%%
% Face integrals
%%%%%%%%%%%%%%%%%%%%%%%%
\section{Evaluation of Face Integrals}
\label{sec:faces}

% Paragraph: data movement in the face terms.
The volume term of~\eqref{eq:sipg} is evaluated cell by cell with sum factorization and touches only the local DoF
block. Face terms require the trace and normal derivative from two adjacent cells and contribute to both residual
blocks. A cell-centric loop reads the neighbors of the current cell, writes only its residual, and therefore evaluates
each interior face twice. A face-centric loop evaluates each face once but writes both residual blocks, requiring
atomics or face coloring. On an unstructured mesh, both organizations use indirect, index-list-driven gathers that
cause warp divergence and fragmented memory transactions~\cite{Kronbichler2017a,fischer2020scalability}. Trace
extraction adds another cost: a face fragment is a lower-dimensional slice of the DoF block, contiguous only for faces
normal to the slowest storage axis and strided for the others. Cartesian meshes align the local coordinate frames and
avoid orientation permutations; their remaining topological irregularity is the non-matching face of adaptive
refinement (\secref{sec:shadow}). Existing implementations reduce these costs through face batching, index renumbering,
and overlap of gather and compute~\cite{Kronbichler2017a,muthing2017high}, but retain the global face list and its
indirect addressing.

% Paragraph: organize face integrals as structured face exchanges.
The block-structured storage removes the face list for the bulk of the mesh. Within a structured block, the faces
normal to each axis $\alpha$ are enumerated arithmetically, and we evaluate them by one dimensionally-split sweep per
axis: a \emph{face-exchange} pass in which each thread owns one interface, extracts the two coincident face fragments
at statically known offsets, evaluates the averages, jumps, and penalty terms of the numerical flux, and writes the two
resulting contributions back, one to each adjacent cell. Neighbor access becomes a direct load at a hard-coded offset
rather than an indirect gather, and each face is evaluated exactly once, without coloring or double evaluation. The
two-sided write is race-free without atomics because of the layer structure of a single pass: within the sweep for axis
$\alpha$, the face at the lower end of a cell writes only the DoF layers $\{0, 1\}$ of that cell and the face at its
upper end only the layers $\{p-1, p\}$, so no two faces of the pass touch a common degree of freedom. The three sweeps
are therefore issued in sequence, one per axis, and it is the ordering between them, rather than atomics, that resolves
the accumulation of the different axes onto shared layers. Because the sweeps are axis-specialized, the strided
extraction pattern of each pass is fixed at compile time, and the block size $B$ can be tuned so that face fragments
span few cache lines. Only the skeleton between macro-blocks retains a residue of indirection because the identity of
the neighboring block must be looked up. On the Cartesian grid, however, no orientation permutation is needed there:
the exchanged fragments combine index-for-index, and the share of such faces shrinks with the block size.

% Paragraph: exact tensor-product factorization of the face operator on Cartesian faces.
On Cartesian faces, the terms of~\eqref{eq:sipg} depend on their arguments only through traces and normal derivatives;
on an axis-aligned face both are tensor products of a one-dimensional endpoint evaluation in the normal direction and
full basis expansions in the tangential directions, and the penalty $\sigma_F$ is constant on each face. The SIPG
contribution of a face $F$ normal to axis $\alpha$, shared by cells $K^-$ and $K^+$, therefore admits the
factorization:
\begin{equation} \label{eq:face_factorized}
    \tilde{f}_{K^s} \gets \tilde{f}_{K^s} + \sum_{s' \in \{-,+\}}
    \bigl( M \otimes \cdots \otimes \underbrace{N^{ss'}}_{\alpha} \otimes \cdots \otimes M \bigr)\, u_{K^{s'}},
    \qquad s \in \{-,+\},
\end{equation}
where $M$ is the 1D mass matrix occupying the tangential slots and the four \emph{face-coupling matrices} $N^{ss'} \in
    \mathbb{R}^{(p+1) \times (p+1)}$ occupy the normal slot; we abbreviate the factor with $N^{ss'}$ in slot $\alpha$
by $\mathcal{N}^{ss'}_\alpha$. The $N^{ss'}$ are assembled once, from rank-one outer
products of the vectors of 1D basis values and (scaled) derivatives at the face endpoint; for instance
\begin{equation} \label{eq:face_coupling}
    N^{--} = \sigma_F\, e\, e^T - \tfrac{1}{2} \bigl( e\, g^T + g\, e^T \bigr),
    \qquad e_i = \psi_i(1), \quad g_i = h^{-1} \psi_i'(1),
\end{equation}
with the analogous expressions, differing only in signs and in which endpoint is evaluated, for the mixed and
$(+,+)$ blocks. Symmetry of SIPG makes the $4(p+1)^2$-entry family $\{N^{ss'}\}$ symmetric as a whole, and on the
Cartesian mesh the \emph{same} four matrices serve every face of a refinement level and every axis; only the tensor
slot moves. The complete face-integration data of the method is thus $4(p+1)^2$ numbers per level, resident in
constant memory or registers; no face quadrature is performed at runtime, and the face kernel reduces to the dense 1D
contractions of~\eqref{eq:face_factorized} on contiguous sub-tiles.

%%%%%%%%%%%%%%%%%%%%%%%%
% Tensor-product evaluation and the Hermite basis
%%%%%%%%%%%%%%%%%%%%%%%%
\subsection{Tensor-Product Evaluation and the Hermite Basis}
\label{sec:hermite}

% Paragraph: the volumetric pipeline on Cartesian cells.
The volume term follows the standard sum-factorized pipeline~\cite{Orszag1980,deville2002highorder,Kronbichler2012},
which on a Cartesian cell requires no quadrature loop at all: with the 1D stiffness and mass matrices $K$ and $M$
precomputed from the basis,
\begin{equation} \label{eq:eval_pipeline}
    \tilde{f}_K = h^{d-2} \bigl( K \otimes M \otimes M + M \otimes K \otimes M + M \otimes M \otimes K \bigr)\, u_K,
\end{equation}
i.e.\ the same seven-contraction form benchmarked for the Laplacian in~\cite{wichrowski2026DSS}. Each contraction acts
along one tensor axis and maps directly onto the blocked storage layout: the $B$ and $n_e$ axes are vectorized over,
and the 1D matrices are small dense factors suited to tile-based hardware. The face terms are the
contractions~\eqref{eq:face_factorized} on the face-adjacent sub-tiles.

% Paragraph: the face-data problem of a generic nodal basis.
The choice of the 1D basis decides how much of a neighbor's DoF block the face contractions must read, through the
sparsity of the endpoint vectors $e$ and $g$ in~\eqref{eq:face_coupling}. With a Gauss--Lobatto nodal basis the value
vector $e$ is a unit vector, so the trace is the boundary layer of the DoF block, whereas the derivative vector $g$ is
dense: every 1D Lagrange function has a nonvanishing derivative at the endpoints, so the coupling matrices $N^{ss'}$
have full rows and columns, $\partial_n u|_F$ depends on the \emph{entire} DoF block of the cell, and the face exchange
must load $(p+1)^d$ values per side to produce $(p+1)^{d-1}$ numbers. We remove this imbalance by a change of basis.

\subsubsection{The Hermite-Type Basis and Two-Layer Face Access}
\label{sec:hermite_basis}

% Paragraph: definition of the 1D basis by endpoint conditions.
Following~\cite{Kronbichler2017a,kronbichler2019multigrid}, we replace the 1D Lagrange basis by a \emph{Hermite-type}
basis $\{\psi_0, \dots, \psi_p\} \subset \mathbb{P}_p([0,1])$, $p \ge 3$, defined by the endpoint conditions
\begin{equation} \label{eq:hermite_conditions}
    \begin{aligned}
        \psi_0(0) & = 1, & \psi_0'(0) & = 0, & \qquad \psi_1(0) & = 0, & \psi_1'(0) & = 1,                         \\
        \psi_i(0) & = 0, & \psi_i'(0) & = 0, & \qquad \psi_i(1) & = 0, & \psi_i'(1) & = 0, \qquad 2 \le i \le p-2,
    \end{aligned}
\end{equation}
with the mirrored conditions $\psi_{p-1}'(1) = 1$, $\psi_p(1) = 1$ (and all other endpoint values and derivatives of
$\psi_0, \psi_1, \psi_{p-1}, \psi_p$ equal to zero) at the right endpoint. The four endpoint functions are the cubic
Hermite polynomials augmented by higher-order bubbles enforcing~\eqref{eq:hermite_conditions}; the $p-3$ interior
functions may be chosen as Lagrange functions on interior Gauss--Lobatto nodes multiplied by the quartic bubble
$x^2(1-x)^2$, which annihilates their endpoint values and derivatives. For a 1D expansion $u(x) = \sum_i u_i
    \psi_i(x)$ the conditions~\eqref{eq:hermite_conditions} give the defining property of the basis:
\begin{equation} \label{eq:hermite_traces_1d}
    u(0) = u_0, \qquad u'(0) = u_1, \qquad u'(1) = u_{p-1}, \qquad u(1) = u_p,
\end{equation}
i.e.\ the value and the derivative at each endpoint are \emph{coefficients} of the expansion, read off without any
contraction.

% Paragraph: show that tensor-product traces occupy two layers.
The property tensorizes. Consider a 3D cell with coefficients $u_{ijk}$ and, without loss of generality, the face $F =
    \{x_1 = 0\}$. Inserting~\eqref{eq:hermite_traces_1d} into the tensor-product expansion
\begin{equation}
    \label{eq:hermite_expansion_3d}
    u(x) = \sum_{ijk} u_{ijk}\,
    \psi_i(x_1) \psi_j(x_2) \psi_k(x_3)
\end{equation}
yields
\begin{equation} \label{eq:hermite_traces_3d}
    u\big|_F = \sum_{j,k} u_{0jk}\, \psi_j(x_2) \psi_k(x_3),
    \qquad
    \partial_n u\big|_F = \sum_{j,k} u_{1jk}\, \psi_j(x_2) \psi_k(x_3),
\end{equation}
since every $\psi_i$ with $i \ge 2$ contributes neither value nor derivative at $x_1 = 0$. The trace is the DoF layer
$i = 0$ and the normal derivative is the layer $i = 1$ (up to the scaling $h^{-1}$ of the reference-to-physical map);
the analogous statement holds for every face, with the layers $i \in \{p-1, p\}$ on the far side. All face data of a
cell thus reside in the two layers adjacent to each face, namely the face-adjacent sub-tiles of the blocked
storage layout.

% Paragraph: the coupling matrices collapse to a symmetric 4x4 on the face layers.
In this basis the endpoint vectors of~\eqref{eq:face_coupling} become unit vectors, $e = e_p$ and $g = h^{-1} e_{p-1}$
at the right endpoint (and $e_0$, $h^{-1} e_1$ at the left), so the four coupling matrices $N^{ss'}$ lose all but a
handful of entries: the face operator acts only on the two face-adjacent layers of each side. Writing $a^\pm$ for the
value layer and $b^\pm$ for the derivative layer of the two cells at a common face, the normal-direction coupling
of~\eqref{eq:face_factorized} collapses to the symmetric $4 \times 4$ stencil
\begin{equation} \label{eq:hermite_face_stencil}
    \begin{pmatrix} f_a^- \\ f_b^- \\ f_a^+ \\ f_b^+ \end{pmatrix}
    \gets
    \begin{pmatrix} f_a^- \\ f_b^- \\ f_a^+ \\ f_b^+ \end{pmatrix}
    +
    \begin{pmatrix}
        \sigma_F  & -\beta & -\sigma_F & -\beta \\
        -\beta    & 0      & \beta     & 0      \\
        -\sigma_F & \beta  & \sigma_F  & \beta  \\
        -\beta    & 0      & \beta     & 0
    \end{pmatrix}
    \begin{pmatrix} a^- \\ b^- \\ a^+ \\ b^+ \end{pmatrix},
    \qquad \beta = \frac{1}{2h},
\end{equation}
applied under the tangential mass contractions $M \otimes M$ of~\eqref{eq:face_factorized}. Both cells are mapped from
the reference interval with the same orientation, so $b^-$ and $b^+$ are derivatives along $+\alpha$ rather than along
each cell's outward normal; this is why the two derivative rows of~\eqref{eq:hermite_face_stencil} coincide, and why no
outward-normal sign appears anywhere in the stencil. Its entries are then read off directly: the value rows carry
$\sigma_F \jump{u} - \avg{\partial_n u}$ and the derivative rows $-\tfrac{1}{2} h^{-1} \jump{u}$, with the factor
$\tfrac{1}{2}$ of the average and the $h^{-1}$ of the reference-to-physical map collected into $\beta$. A face
transaction reads and writes $2(p+1)^{d-1}$ values per side instead of $(p+1)^d$, reducing the neighbor traffic
by a factor $(p+1)/2$. It consists of dense contractions on contiguous sub-tiles, preserving the static addressing
of the face-exchange sweeps.

% Paragraph: cost of the basis change.
The volumetric operator~\eqref{eq:eval_pipeline} retains its size and operation count; the matrices $K$ and $M$ are
tabulated for $\{\psi_i\}$. The transform between the Lagrange and Hermite-type bases is a fixed, well-conditioned
$(p+1) \times (p+1)$ matrix applied per tensor direction, needed only when interfacing with nodal data (e.g.\ for
output or right-hand-side evaluation); the discrete space, and hence the discretization, is unchanged. The
block-diagonal mass matrix of DG loses its diagonality in the Hermite basis, but none of our algorithms relies on it.

%%%%%%%%%%%%%%%%%%%%%%%%
% Shadow cells for non-matching faces
%%%%%%%%%%%%%%%%%%%%%%%%
\subsection{Shadow Cells for Non-Matching Faces}
\label{sec:shadow}

% Paragraph: explain how non-matching faces break the uniform face exchange.
Local adaptive refinement produces, along every refinement edge, faces at which a single coarse cell abuts $2^{d-1}$
finer neighbors. Throughout we assume the mesh is \emph{2:1 balanced}: adjacent cells differ by at most one refinement
level, an invariant maintained by standard octree mesh
frameworks~\cite{burstedde2011p4est,bangerth2012algorithms}.\footnote{We assume 2:1 balance throughout because it is
    what we implemented and measured, not because the construction appears to require it. We see no obstruction to lifting
    it: at a jump of more than one level the shadow layer would have to be nested to the depth of the jump, so that every
    active face again meets a partner at its own level, and the local multigrid method of \secref{sec:local_mg} should
    carry over unchanged. We have not tested this.} Non-matching faces break the uniform data flow established above. The
face integral on a non-matching interface couples the coarse trace, restricted to a subface, with a fine trace; its
evaluation requires subface quadrature with shifted and scaled 1D rules, per-subface interpolation matrices, and a
branch distinguishing which side is coarse. This requires a second family of face kernels with irregular addressing and
warp divergence. Classical matrix-free DG implementations accept this cost and maintain dedicated data structures and
kernels for the non-matching case~\cite{Kronbichler2017a}; mortar-type treatments in the spectral-element
literature~\cite{kopriva1996conservative,kopriva2002mortar} similarly route the coupling through auxiliary interface
spaces.

% Paragraph: introduce shadow cells as the remedy for non-matching faces.
We eliminate non-matching faces from the face loop. Every active cell whose face lies on the \emph{coarse side} of a
refinement edge is artificially refined, and its $2^d$ children are allocated in storage as \emph{shadow cells}
(Figure~\ref{fig:shadow_dg}). We refer to these auxiliary cells as \emph{shadows} because they serve exclusively to
evaluate interface terms. Like a shadow, they carry only an outline: these cells do not participate in volumetric
integration, so on a deformed mesh only the geometry of their periphery is needed; on the Cartesian mesh the cell size
$h$ already determines it, and the shadows carry no geometric data at all. The shadows are structurally identical to
active cells in DoF count and block layout but carry no independent degrees of freedom: their content is always derived
from the parent by polynomial prolongation. With the shadows in place, every face of every active fine cell finds a
\emph{matching} partner at its own level: either a genuine fine neighbor or a shadow child. The non-matching faces of
the original mesh are skipped; their contributions flow through the cell-shadow faces instead.

\begin{figure}[htbp]
    \centering
    % Shadow children on the coarse side of a 2D h-adaptive Cartesian mesh.
% The active refined cells form the same patch as in fig_dss_shadow.tex, while
% the gray overlays mark virtual children of coarse cells along its boundary.
\begin{tikzpicture}[
                scale=0.82,
                coarse/.style={draw=black, thin},
                fine/.style={draw=black!70, line width=0.4pt},
                shadowfill/.style={fill=gray!35},
                shadow/.style={draw=black!65, densely dashed, line width=0.45pt},
                interface/.style={red!75!black, line width=1.45pt,
                    shorten >=2.5pt, shorten <=2.5pt},
                label/.style={font=\footnotesize\bfseries},
                legend/.style={font=\scriptsize, anchor=west}]
        % Active coarse cells whose virtual children cover the coarse side of the interface.
        \def\shadowparents{%
                0/2,1/2,6/2,7/2,%
                2/3,3/3,4/3,5/3}

        % Coarse background grid, including one intact layer above the shadow band.
        \draw[coarse] (0,0) rectangle (8,5);
        \foreach \i in {1,...,7} { \draw[coarse] (\i,0) -- (\i,5); }
        \foreach \j in {1,...,4} { \draw[coarse] (0,\j) -- (8,\j); }

        % Fine active cells: two globally refined layers and a raised central patch.
        \foreach \c in {0,...,7} {
                \draw[fine] (\c+0.5,0) -- (\c+0.5,2);
                \draw[fine] (\c,0.5) -- (\c+1,0.5);
                \draw[fine] (\c,1.5) -- (\c+1,1.5);
        }
        \draw[fine] (0,1) -- (8,1);
        \foreach \c in {2,...,5} {
                \draw[fine] (\c+0.5,2) -- (\c+0.5,3);
                \draw[fine] (\c,2.5) -- (\c+1,2.5);
        }

        % Overlay the virtual shadow children on coarse cells adjacent to the fine patch.
        \foreach \c/\r in \shadowparents {
                \fill[shadowfill] (\c,\r) rectangle (\c+1,\r+1);
                \draw[coarse] (\c,\r) rectangle (\c+1,\r+1);
                \draw[shadow] (\c+0.5,\r) -- (\c+0.5,\r+1);
                \draw[shadow] (\c,\r+0.5) -- (\c+1,\r+0.5);
        }

        % Original non-matching interface; its two halves become matching fine-shadow faces.
        \foreach \x in {0,1,6,7} { \draw[interface] (\x,2) -- (\x+1,2); }
        \foreach \x in {2,3,4,5} { \draw[interface] (\x,3) -- (\x+1,3); }
        \draw[interface] (2,2) -- (2,3);
        \draw[interface] (6,2) -- (6,3);

        % Labels and compact legend.
        \node[label] at (4,0.22) {active fine cells};
        \node[label] at (4,4.72) {active coarse cells};
        \begin{scope}[shift={(8.35,2.25)}]
                \fill[shadowfill] (0,0) rectangle (0.28,0.22);
                \draw[coarse] (0,0) rectangle (0.28,0.22);
                \draw[shadow] (0.14,0) -- (0.14,0.22);
                \node[legend] at (0.38,0.11) {shadow children};
                \draw[interface] (0,0.65) -- (0.57,0.65);
                \node[legend] at (0.68,0.65) {matching fine--shadow faces};
        \end{scope}
\end{tikzpicture}
    \caption{Shadow cells at a refinement edge. The gray overlays are virtual children of the active coarse cells
        adjacent to the refined patch. They receive the prolonged coarse-cell solution, so each active fine cell meets
        a same-level shadow across the red interface. The face loop therefore processes the red segments as matching
        fine--shadow faces and skips the original non-matching interfaces. When an assembled residual is required, the
        shadow restriction assembles contributions accumulated on the shadows into their active parents.}
    \label{fig:shadow_dg}
\end{figure}
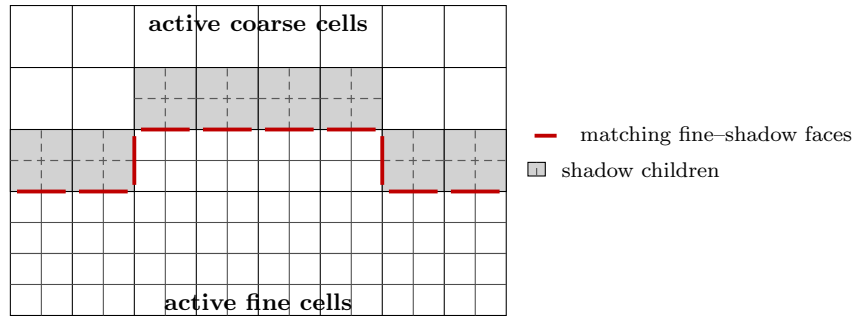

% Paragraph: the algorithm, narrative form, with the two parts of the split on their own cell sets.
The shadow-cell operator application evaluates the two parts of~\eqref{eq:operator_split} on different cell sets. The
volume kernel runs only over active cells because applying it to shadows would duplicate their parent's contribution.
Before the face sweeps, the shadow refresh embeds each parent into its children by $\bigotimes_i\mathcal{P}_{b_i}$,
where $b = b(c) \in \{0,1\}^d$ is the binary multi-index identifying the position of the child $c$ within its parent.
The sweeps then process all matching faces, including fine-shadow faces, and skip the original non-matching interfaces;
we write $\mathcal{F}^i_{\text{exch}}$ for the matching faces normal to axis $i$ with at least one active side,
excluding faces between shadows of distinct parents. Every invocation receives same-level, same-size fragments and need
not distinguish active from shadow cells. When an assembled residual is required, shadow assembly adds the shadow
contributions to their parents by $\bigotimes_i\mathcal{P}_{b_i}^T$, after which the shadows are cleared. This clearing
needs no pass of its own: the volume kernel \emph{writes} its output rather than accumulating into it, and since it
scales every cell by a factor that is zero on shadows, it leaves the shadow blocks zeroed and the whole residual buffer
in a known state for the accumulating face sweeps. The next operator application therefore performs the clear
implicitly, and no separate pass over the shadow layer is needed.

% Paragraph: which transfers survive in which caller.
Algorithm~\ref{alg:shadow_face_loop} summarizes these steps. It is the self-contained form of an operator application:
one that starts from arbitrary active data and returns an assembled active-mesh residual. Inside the Krylov iteration
of \secref{sec:dg_pairing} \emph{both} transfers are dropped, and only the volume kernel and the face sweeps remain.
Shadow assembly is dropped because the iteration carries split duals and never needs the assembled form; the shadow
refresh is dropped because its iterates are consistent by construction and the vector updates preserve that property,
so the shadows already hold their parents' prolongations on entry to every application. Multigrid performs shadow
assembly as part of its level restriction and restores consistency with a single shadow refresh at the exit of the
V-cycle. All of these transfers act only on the one-cell-deep shadow layer and therefore have refinement-edge surface
cost.

\begin{algorithm}[H]
    \caption{Shadow-cell evaluation of $A u$ on an adaptively refined mesh}
    \label{alg:shadow_face_loop}
    \begin{algorithmic}[1]
        \Require Solution vector $u$, residual output $\tilde{f}$, shadow parents $\mathcal{T}_{\text{sh}}$ with children $\text{ch}(K)$.
        \For{\textbf{each} active cell $K \in \mathcal{T}_{\text{act}}$ \textbf{in parallel}}
        \State $\tilde{f}|_K \gets A^{\text{vol}} u|_K$ \Comment{Volume kernel~\eqref{eq:eval_pipeline}}
        \EndFor
        \For{\textbf{each} shadow parent $K \in \mathcal{T}_{\text{sh}}$ \textbf{in parallel}}
        \State \label{ln:shadow_refresh} $u|_{K_c} \gets \bigl( \bigotimes_{i=1}^d \mathcal{P}_{b_i} \bigr)\, u|_K$ \textbf{for all} $K_c \in \text{ch}(K)$ \Comment{Shadow refresh}
        \EndFor
        \For{axis $i = 1, \dots, d$}
        \For{\textbf{each} face $F \in \mathcal{F}^i_{\text{exch}}$ with cells $K^-, K^+$ \textbf{in parallel}}
        \State $\tilde{f}|_{K^s} \gets \tilde{f}|_{K^s} + \sum_{s' \in \{-,+\}} \mathcal{N}^{ss'}_i\, u|_{K^{s'}}$,
        \ $s \in \{-,+\}$ \Comment{Face exchange~\eqref{eq:face_factorized}}
        \EndFor
        \EndFor
        \For{\textbf{each} shadow parent $K \in \mathcal{T}_{\text{sh}}$ \textbf{in parallel}}
        \State \label{ln:shadow_restriction} $\tilde{f}|_K \gets \tilde{f}|_K + \sum_{K_c \in \text{ch}(K)} \bigl( \bigotimes_{i=1}^d \mathcal{P}_{b_i}^T \bigr)\, \tilde{f}|_{K_c}$,\quad
        $\tilde{f}|_{K_c} \gets 0$ \Comment{Shadow restriction}
        \EndFor
    \end{algorithmic}
\end{algorithm}

% Paragraph and theorem: establish that the construction introduces no approximation.
The shadow construction reorganizes the computation without changing the standard SIPG face integrals on the original
non-matching mesh. The following theorem states this equivalence. For a face $F$ of the skeleton, let $a_F(u, v)$
denote the SIPG face form, i.e.\ the contribution of $F$ to the second sum in~\eqref{eq:sipg}, defined for any pair of
functions that are piecewise polynomial on the two sides of $F$. On a non-matching interface, where a coarse cell $K$
abuts fine cells $K_1', \dots, K_{2^{d-1}}'$ across the subfaces $F_c = \partial K_c \cap \partial K_c'$ of its face $F
    = \bigcup_c F_c$, the standard method evaluates $a_{F_c}$ subface by subface, pairing the restriction of the coarse
traces to $F_c$ with the fine traces~\cite{Kronbichler2017a}. Denote by $\mathcal{P}_K \colon \mathbb{Q}_p(K) \to
    \prod_c \mathbb{Q}_p(K_c)$ the parent-to-children polynomial embedding realized by the tensor-product matrices
$\bigotimes_i \mathcal{P}_{b_i}$, and by $K_c$ the shadow child of $K$ adjacent to the subface $F_c$. One convention
has to be fixed before the two methods can be compared: on a fine-shadow face both sides are level-$\ell$ cells, so the
shadow evaluation necessarily uses the fine penalty $\sigma_{F_c} \sim p^2 / h_{F_c}$ with $h_{F_c}$ the fine cell
size. We therefore compare against the standard method run with the same choice, which is also the usual one: the
penalty of a non-matching face is taken from the smaller of the two cells, since it is the fine trace inverse
inequality that dictates the coercivity threshold. The left-hand side of~\eqref{eq:shadow_exact} below is what
Algorithm~\ref{alg:shadow_face_loop} computes, namely the matching-face forms on the fine-shadow faces, with the trial
data on the shadows given by the shadow embedding $\mathcal{P}_K u|_K$ and the test data routed back to $K$ by the
transposed embedding.

\begin{theorem}[Exactness of the shadow-cell face evaluation]
    \label{thm:shadow_exact}
    Let $F = \bigcup_c F_c$ be a non-matching interface between a coarse cell $K$ and fine cells $K_c'$, let $u, v$ be
    arbitrary elements of the broken space $\mathbb{V}_{DG}$ on the active mesh, and let both sides use the fine-level
    penalty $\sigma_{F_c}$ on every subface. Then
    \begin{equation} \label{eq:shadow_exact}
        \sum_{c} a_{F_c}\bigl( \{ u|_{K_c'},\, (\mathcal{P}_K u|_K)|_{K_c} \},\, \{ v|_{K_c'},\, (\mathcal{P}_K v|_K)|_{K_c} \} \bigr)
        \;=\; \sum_{c} a_{F_c}\bigl( \{ u|_{K_c'},\, u|_K \},\, \{ v|_{K_c'},\, v|_K \} \bigr),
    \end{equation}
    the right-hand side being the standard non-matching SIPG contribution of $F$. The identity holds for the computed
    values as well: the same-level quadrature on each fine-shadow face integrates the same integrand over the same
    subface as the dedicated subface rule.
\end{theorem}

\begin{proof}
    The embedding $\mathcal{P}_K$ is exact: $(\mathcal{P}_K w)|_{K_c}$ and $w|_K$ are the \emph{same polynomial},
    restricted to the subdomain $K_c \subset K$. Hence on each subface $F_c$ the trace and the normal derivative of
    the prolonged trial field coincide pointwise with those of $u|_K$, and likewise for the test field:
    \begin{equation*}
        (\mathcal{P}_K u|_K)\big|_{F_c} = u|_K\big|_{F_c}, \qquad
        \partial_n (\mathcal{P}_K u|_K)\big|_{F_c} = \partial_n u|_K\big|_{F_c}.
    \end{equation*}
    The face form $a_{F_c}$ depends on its arguments only through these traces and normal derivatives, so each
    summand on the left of~\eqref{eq:shadow_exact} equals the corresponding summand on the right as an integral over
    $F_c$. Algorithm~\ref{alg:shadow_face_loop} evaluates the left-hand side as follows. On the
    trial side this is the shadow refresh of line~\ref{ln:shadow_refresh}. On the test side, the face kernel on $F_c$ accumulates into the
    shadow child $K_c$ the local functional $w_c \mapsto a_{F_c}(\cdot, w_c)$ expressed in the fine basis of $K_c$;
    the shadow restriction of line~\ref{ln:shadow_restriction} applies $\mathcal{P}_K^T$, and for any coarse test function $v|_K$,
    \begin{equation*}
        \langle \mathcal{P}_K^T \tilde{f}, v|_K \rangle = \langle \tilde{f}, \mathcal{P}_K v|_K \rangle,
    \end{equation*}
    i.e.\ the accumulated coarse residual is the face form tested with the embedded coarse test functions,
    which is the left-hand side of~\eqref{eq:shadow_exact}. For the quadrature claim, note that the fine-level rule
    on the matching face $F_c$ is, after the affine subface map, the shifted and scaled rule that a
    dedicated non-matching kernel applies on $F_c$, and by the trace identities above it is applied to the same
    integrand.
\end{proof}

% Paragraph: operator form of refinement edge contribution.
In operator form, the contribution of the refinement edge to the coarse cell is $\mathcal{P}_K^T \mathcal{F}
    \mathcal{P}_K$, with $\mathcal{F}$ the matching-face flux operator of the fine level. This Galerkin projection
preserves the symmetry of SIPG.

% Paragraph: shadow-shadow faces and the cost of the shadow layer.
Once the shadows are in place, a face of the fine level has both sides shadowed in two situations, and they must be
distinguished. Faces between \emph{siblings}, which are shadow children of one common parent, may be skipped or
processed. Because their data are restrictions of one parent polynomial, the trial jump vanishes and the remaining
consistency contributions cancel under shadow assembly against single-valued coarse test functions; the same
cancellation holds in the duality pairings of \secref{sec:dg_pairing}. Processing these faces therefore leaves the
assembled result unchanged, while skipping them is an optional optimization. Faces between shadows of \emph{different}
parents must, by contrast, be excluded. Such a face arises wherever two adjacent coarse cells both border the refined
region and therefore both cast shadows; it is the fine-level duplicate of the coarse--coarse face between the two
parents, which the face loop already visits at the coarse level. Its trial jump does not vanish, and processing it
would add a second copy of that flux, evaluated with the fine cell size and the fine penalty. The face selection is
thus: all matching faces with at least one active cell, plus, optionally, sibling-shadow faces. The shadow layer adds
$2^d$ cells per coarse cell on a refinement edge and hence surface-order storage and traffic. In return, it avoids
subface quadrature, refinement-case branches, and irregular face kernels. Shadows are allocated as ordinary fine-level
cells, adding blocks without changing their internal layout. Their cost is already contained in the adaptive-depth
measurements of \secref{sec:dg_results_throughput}, whose rates are normalized by active degrees of freedom while the
kernels run over the shadow layer as well.

% Paragraph: the two-layer structure of the Hermite basis carries over to the shadows, and why we do not exploit it.
The two-layer structure of \secref{sec:hermite_basis} carries over to the shadow transfers. A shadow child shares the
face in question with its parent, so along the normal axis the polynomial prolongation inherits the trace
property~\eqref{eq:hermite_traces_1d}: the child's value coefficient on that face is the parent's, and its derivative
coefficient is the parent's halved by the reference-to-physical scaling of the subdivided cell, i.e.\ the first two
rows of $\mathcal{P}_{b_i}$ are $e_0$ and $\tfrac{1}{2} e_1$. Only the two face-adjacent layers of the parent therefore
enter the data a face kernel reads from a shadow, and the shadow restriction returns contributions to those same
layers; the polynomial prolongation remains dense in the tangential directions, where the child covers half the
parent's extent. We do not exploit this sparsity. The shadow refresh processes whole cells, and all $2^d$ children of a
parent are allocated and transferred rather than only the $2^{d-1}$ adjacent to the refinement edge. Some of the
resulting cells are never read, but the transfer is then bit-for-bit the ordinary level prolongation of
\secref{sec:local_mg} applied to a larger cell list, with no special case and no separate kernel. As elsewhere in the
construction, redundant storage buys a uniform loop.

%%%%%%%%%%%%%%%%%%%%%%%%
% Primal-dual pairings on shadowed storage: no shadow assembly in the Krylov loop
%%%%%%%%%%%%%%%%%%%%%%%%
\section{Duality Pairings on Shadowed Storage}
\label{sec:dg_pairing}

% Paragraph: introduce shadow redundancy and the primal-dual question.
Shadow cells introduce redundant storage into DG. If Algorithm~\ref{alg:shadow_face_loop} omits the final shadow
restriction, the refinement-edge contribution to a coarse cell remains on its shadow; we call this representation
\emph{split}. To avoid shadow assembly after every operator application, we use the primal--dual formulation for
redundant active-and-shadow storage~\cite{wichrowski2026DSS}: conjugate gradients pair a split \emph{dual} vector with
a consistent \emph{primal} vector. Bilinearity permits the Krylov iteration to use the split representation, with
shadow assembly fused into the multigrid level restriction (\secref{sec:local_mg}).

% Paragraph: define the extended space, embedding, constraint, and split duals.
Let $\mathbb{V}_{\text{ext}}$ denote the \emph{extended} block space: one DoF block per active cell and per shadow
cell. The \emph{shadow embedding}
\begin{equation} \label{eq:shadow_gather}
    \begin{aligned}
        \mathcal{G}_{\text{sh}} & \colon \mathbb{V}_{DG} \to \mathbb{V}_{\text{ext}},                                                        \\
        (\mathcal{G}_{\text{sh}} u)\big|_K
                                & = u\big|_K,
                                &                                                                      & \text{($K$ active)},                \\
        (\mathcal{G}_{\text{sh}} u)\big|_{K_c}
                                & = \bigl( \textstyle\bigotimes_i \mathcal{P}_{b_i} \bigr)\, u\big|_K,
                                &                                                                      & \text{($K_c$ shadow child of $K$)}.
    \end{aligned}
\end{equation}
implements the shadow refresh of line~\ref{ln:shadow_refresh} of Algorithm~\ref{alg:shadow_face_loop}, together with the identity on active
cells. We call a vector $u \in
    \mathbb{V}_{\text{ext}}$ \emph{consistent} (primal) if it lies in the image of $\mathcal{G}_{\text{sh}}$, i.e.\ if
its shadow blocks are the polynomial prolongations of their parents; consistency is the fixed-point condition $u = C_{\text{sh}}
    u$ of the idempotent constraint operator that refreshes every shadow from its parent. Dual vectors $\tilde{f}
    \in \mathbb{V}_{\text{ext}}^*$ are
split residuals: local integrals whose shadow blocks hold contributions belonging to the parent's test functions.
Their \emph{shadow assembly} is the transpose of the embedding,
\begin{equation} \label{eq:shadow_assembly}
    \mathcal{G}_{\text{sh}}^T \colon \mathbb{V}_{\text{ext}}^* \to \mathbb{V}_{DG}^*,
    \qquad
    (\mathcal{G}_{\text{sh}}^T \tilde{f})\big|_K
    = \tilde{f}\big|_K + \sum_{K_c \in \text{ch}(K)} \bigl( \textstyle\bigotimes_i \mathcal{P}_{b_i}^T \bigr)\, \tilde{f}\big|_{K_c},
\end{equation}
which is the shadow restriction of line~\ref{ln:shadow_restriction} of Algorithm~\ref{alg:shadow_face_loop} (with the sum empty on cells
without shadows). A dual vector is a valid representative of a functional on $\mathbb{V}_{DG}$ through its shadow assembly
alone; in particular the right-hand side is represented by the vector $\tilde{b}$ that carries the local load
integrals on active cells and is \emph{zero on all shadows}.

% Paragraph: the split operator on the extended storage.
Evaluating the split~\eqref{eq:operator_split} on the extended storage and stopping
Algorithm~\ref{alg:shadow_face_loop} before shadow assembly defines
\begin{equation} \label{eq:split_operator}
    \tilde{A} \colon \mathbb{V}_{\text{ext}} \to \mathbb{V}_{\text{ext}}^*,
    \qquad
    \tilde{A} = A^{\text{vol}} + A^{\text{face}},
\end{equation}
for the resulting \emph{shadowed operator}, with the two parts on their respective cell sets: $A^{\text{vol}}$ over
the active cells alone and $A^{\text{face}}$ over the face selection of \secref{sec:shadow}: all matching faces with
at least one active cell, including fine-shadow ones, and, harmlessly, the sibling-shadow faces. Because shadows
receive face contributions but no volume contribution, their blocks contain data belonging to the parent's test
functions. Applied to a consistent vector, $\tilde{A}$ produces a split dual: active cells hold their volume term plus the fluxes of their
matching faces, and shadows hold the refinement-edge fluxes destined for their parents. In operator form,
Theorem~\ref{thm:shadow_exact} gives the Galerkin identity
\begin{equation} \label{eq:galerkin_identity}
    \mathcal{G}_{\text{sh}}^T\, \tilde{A}\, \mathcal{G}_{\text{sh}} = A,
\end{equation}
with $A$ the SIPG operator on the active mesh: applying shadow assembly to the split result of the shadowed evaluation
of a consistent vector yields the exact active-mesh residual.

% Lemma: the pairing identity.
The duality pairings of a Krylov iteration do not require shadow assembly.

\begin{lemma}[Duality pairing on shadowed storage]
    \label{lem:shadow_pairing}
    Let $\tilde{f} \in \mathbb{V}_{\text{ext}}^*$ be a split dual vector and $u = \mathcal{G}_{\text{sh}}
        u_{\text{act}}$ a consistent primal vector. Then the blockwise Euclidean pairing over \emph{all} cells, active
    and shadow, equals the active-mesh pairing of the assembled functional:
    \begin{equation} \label{eq:shadow_pairing}
        \langle \tilde{f}, u \rangle
        \;=\;
        \langle \tilde{f}, \mathcal{G}_{\text{sh}} u_{\text{act}} \rangle
        \;=\;
        \langle \mathcal{G}_{\text{sh}}^T \tilde{f}, u_{\text{act}} \rangle .
    \end{equation}
\end{lemma}

\begin{proof}
    Immediate from the definition of the transpose. The shadow summands pair $\tilde{f}|_{K_c}$ with the prolonged
    parent data, and $\sum_c \langle \tilde{f}|_{K_c},
        \mathcal{P}_K u|_K \rangle = \langle \sum_c \mathcal{P}_K^T \tilde{f}|_{K_c}, u|_K \rangle$ is the
    parent's share of the assembled pairing.
\end{proof}

% Paragraph: exact active energy pairing on extended storage.
In particular, for consistent $u, v$ and $\tilde{q} = \tilde{A} \mathcal{G}_{\text{sh}} u_{\text{act}}$,
combining~\eqref{eq:galerkin_identity} with~\eqref{eq:shadow_pairing} gives $\langle \tilde{q}, v \rangle = \langle A
    u_{\text{act}}, v_{\text{act}} \rangle$: the exact energy pairing of the active mesh, computed as one flat dot product
over the extended storage, with no shadow assembly, no communication, and no distinction between active and shadow
blocks.

% Paragraph: the Krylov loop on split data.
The resulting preconditioned conjugate gradient iteration operates on the shadowed storage. The iterate $u_k$, search
direction $p_k$, and preconditioned residual $z_k$ are consistent primal vectors, while the residual $\tilde{r}_k$ and
the operator output $\tilde{q}_k = \tilde{A} p_k$ are split duals. Every scalar of the iteration, including $\alpha_k =
    \langle \tilde{r}_k, z_k \rangle / \langle \tilde{q}_k, p_k \rangle$, $\beta_k = \langle \tilde{r}_{k+1}, z_{k+1}
    \rangle / \langle \tilde{r}_k, z_k \rangle$ and the monitored quantity $\langle \tilde{r}_k, z_k \rangle$, pairs one
dual with one primal vector and is therefore computed blockwise by Lemma~\ref{lem:shadow_pairing}. The vector updates
preserve both structures by linearity: $u_{k+1} = u_k + \alpha_k p_k$ is a linear combination of consistent vectors and
hence consistent because its shadows are automatically the prolongations of its parents, so no shadow refresh is needed
inside the loop, and $\tilde{r}_{k+1} = \tilde{r}_k - \alpha_k \tilde{q}_k$ is a linear combination of split duals
whose shadow assembly, by linearity of $\mathcal{G}_{\text{sh}}^T$, is the correct active-mesh residual. Operator
applications in the Krylov loop therefore end with the face sweeps and omit the shadow assembly of
Algorithm~\ref{alg:shadow_face_loop}.

% State the preconditioner conditions required for equivalence.
For the fixed symmetric positive-definite multigrid preconditioner, equivalence with preconditioned conjugate gradients
on the assembled active-mesh system requires two representation properties. First, its output $z = P(\tilde{r})$ must
be consistent. This requires one shadow refresh at the exit of the V-cycle, which can be fused with its final level
prolongation. No other shadow refresh is needed because the operator application and vector updates preserve
consistency. Second, the preconditioner must depend on the split residual only through its shadow assembly, i.e.,
\[
    P = \mathcal{G}_{\text{sh}} P_{\text{act}} \mathcal{G}_{\text{sh}}^T
\]
for some preconditioner $P_{\text{act}}$ of the active-mesh system. The local multigrid method of \secref{sec:local_mg}
satisfies this condition: its level restriction transfers the shadow-block contributions with the remaining level
residual into the coarser right-hand sides, realizing~\eqref{eq:shadow_assembly} level by level. Both conditions
therefore hold, and with the Galerkin identity~\eqref{eq:galerkin_identity} shadowed CG produces the same iterates as
preconditioned CG on the assembled active-mesh system.

% Identify the valid residual measure.
The Euclidean norm of a split dual vector is not the norm of the assembled residual; convergence must therefore be
monitored through $\langle \tilde{r}_k, z_k \rangle$ or another dual--primal pairing, not through $\|\tilde{r}_k\|_2$.

%%%%%%%%%%%%%%%%%%%%%%%%
% Local multigrid for DG
%%%%%%%%%%%%%%%%%%%%%%%%
\section{Local Multigrid with Adaptively Refined Meshes}
\label{sec:local_mg}

% Paragraph: solver and preconditioner.
We precondition the conjugate gradient method of \secref{sec:dg_pairing} by a geometric multigrid V-cycle on the mesh
refinement hierarchy, taken from~\cite{wichrowski2026MG}, where it is developed for continuous elements on persistent
cell-wise storage. Shadow cells represent the coupling at refinement edges within the operator's face loop.

\subsection{The Local Multigrid Method}
\label{sec:local_mg_intro}

% Paragraph: define the refinement tree, active mesh, and levels.
Recursive subdivision of Cartesian cells into $2^d$ children produces an octree refinement
tree~\cite{burstedde2011p4est,bangerth2012algorithms}. Its leaves form the \emph{active mesh}
$\mathcal{T}_{\text{act}}$, whose adjacent cells may belong to different levels. For $\ell\in\{0,\dots,L\}$, let
$\mathcal{T}_\ell$ contain the level-$\ell$ cells, active or further refined; the level meshes
$\mathcal{T}_0\sqsubset\cdots\sqsubset\mathcal{T}_L$ and their broken spaces are nested. The \emph{refinement edge}
$E_\ell$ separates level-$\ell$ cells from coarser active cells and carries both non-matching faces and inter-level
coupling.

% Paragraph: why global multigrid fails; local smoothing and its lineage.
Classical geometric multigrid smooths on a sequence of uniformly refined meshes~\cite{Hackbusch85,Bramble93}. Embedding
an adaptive mesh into such a hierarchy would make each level's cost depend on the globally refined mesh rather than the
locally refined region. Local multigrid instead smooths only the level-$\ell$ cells and leaves the rest of the domain
to coarser levels~\cite{brandt1977mlat,mccormick1986fac,JanssenKanschat11}. Matrix-free realizations exist for
continuous and discontinuous elements~\cite{kronbichler2019multigrid,fehn2020hybrid}. The work of a local V-cycle is
proportional to the number of cells summed over the adaptive levels.

% Paragraph: level decomposition and the iteration structure.
Because every active cell belongs to one level, the active broken space is the direct sum of per-level subspaces. No
DoF is shared between levels, unlike at refinement edges of conforming spaces. The level-$\ell$ right-hand side equals
the active residual on active level-$\ell$ cells and vanishes on refined cells. During descent, a shadow contribution
is added to the right-hand side of its coarser parent by shadow assembly, fused with the level restriction.
Algorithm~\ref{alg:dg_local_vcycle} summarizes the remaining smoothing, residual, level-restriction, coarse-correction,
and level-prolongation steps. On exit, level prolongation also refreshes the shadow layer to satisfy the consistency
condition of \secref{sec:dg_pairing}. Level prolongation applies $\bigotimes_i\mathcal{P}_{b_i}$ parent by parent, and
level restriction applies its transpose child by child. These are Galerkin transfers between nested broken spaces and
require neither constraints nor interface weighting.

\begin{algorithm}[H]
    \caption{Local multigrid $V$-cycle on level $\ell$: $MG_{\text{local}}(A_\ell, b_\ell, u_\ell, \ell)$.}
    \label{alg:dg_local_vcycle}
    \begin{algorithmic}[1]
        \Require Level $\ell$, level iterate $u_\ell$ (zero on entry), level right-hand side $b_\ell$ (carrying the active residual of level $\ell$).
        \If{$\ell = 0$}
        \State $u_0 \gets A_0^{-1} b_0$ \Comment{Coarse grid solve}
        \State \Return $u_0$
        \EndIf
        \State $u_\ell \gets \mathcal{S}_\ell^{\nu_1}(A_\ell, b_\ell, u_\ell)$ \Comment{Pre-smoothing on level-$\ell$ cells; shadows zero and frozen}
        \State $r_\ell \gets b_\ell - A_\ell u_\ell$ \Comment{Level residual}
        \State $b_{\ell-1} \gets b_{\ell-1} + \mathcal{R}_\ell^{\ell-1} r_\ell$ \Comment{Level restriction into the coarser right-hand side}
        \State $u_{\ell-1} \gets MG_{\text{local}}(A_{\ell-1}, b_{\ell-1}, u_{\ell-1}, \ell - 1)$ \Comment{Recursive coarse-grid correction}
        \State \label{ln:level_prolong} $u_\ell \gets u_\ell + \mathcal{P}_{\ell-1}^\ell u_{\ell-1}$ \Comment{Level prolongation; refreshes shadows}
        \State $u_\ell \gets \mathcal{S}_\ell^{\nu_2}(A_\ell, b_\ell, u_\ell)$ \Comment{Post-smoothing}
        \State \Return $u_\ell$
    \end{algorithmic}
\end{algorithm}

\subsection{Smoother}
\label{sec:mg_smoother}

% State the admissible smoother class and the choice used here.
The local V-cycle is not restricted to a particular smoother. It requires only that a level-$\ell$ smoothing step
update the level-$\ell$ unknowns while leaving the shadow values fixed. For simplicity, we use cell-wise block-Jacobi
in the spirit of~\cite{bastian2019matrix,WitteArndtKanschat21}, taking as the block the cell-diagonal $(p+1)^d \times
    (p+1)^d$ part of $A_\ell$. On a Cartesian cell that block inherits the tensor structure of the level operator: the two
faces of the cell normal to a given axis contribute on disjoint layers of that axis and under the mass matrix in the
other two, so the block is the Kronecker sum
\begin{equation} \label{eq:cell_block}
    D_K = \hat{A} \otimes M \otimes M + M \otimes \hat{A} \otimes M + M \otimes M \otimes \hat{A},
\end{equation}
with $\hat{A}$ the 1D stiffness matrix plus the own-side $2 \times 2$ sub-blocks of the two face stencils. Its inverse
is therefore available exactly by fast diagonalization~\cite{Lynch1964,deville2002highorder}: one generalized
eigenproblem $\hat{A} q = \lambda M q$ per level yields $D_K^{-1} = (Q \otimes Q \otimes Q)\, \Lambda^{-1}\, (Q \otimes
    Q \otimes Q)^T$ with $\Lambda_{ijk} = \lambda_i + \lambda_j + \lambda_k$, so an application is six one-dimensional
contractions and one diagonal scaling, with the same sum-factorized shape as the operator itself. The approximation lies in
the choice of block, not in its inversion: $D_\ell$ discards the off-cell face couplings, and on cells with a domain
boundary face we use the interior stencil in place of the Nitsche one, which a smoother is free to do.
Let $D_\ell^{-1}$ denote the resulting block-Jacobi preconditioner. Its application is cell-local and requires no
inter-cell communication; communication within the V-cycle occurs only through the face terms of the level operators
and the inter-level transfers.

% Derive the three Chebyshev weights from the target spectral interval.
We accelerate block-Jacobi by a Chebyshev polynomial over a target interval $[a,b]$ containing the relevant eigenvalues
of $D_\ell^{-1}A_\ell$, where $0<a<b$. The implementation admits an arbitrary polynomial degree; degree three gave the
best time to solution among the degrees we tried, and every result reported in \secref{sec:results} uses it, applied
once before and once after the coarse-grid correction ($\nu_1 = \nu_2 = 1$ in Algorithm~\ref{alg:dg_local_vcycle}). We
therefore state the construction for that degree.
%% NOTE: degree three is the measured configuration; the solver's own command-line default is 2. Do not
%% change the text to 2; fix the default instead.
Define
\begin{equation}
    c=\frac{a+b}{2},
    \qquad
    d=\frac{b-a}{2},
    \qquad
    \omega_j=
    \frac{1}{c-d\cos\!\left(\frac{(2j-1)\pi}{6}\right)},
    \quad j=1,2,3.
    \label{eq:chebyshev_weights}
\end{equation}
One smoothing application is implemented in factorized form as three nonstationary Richardson updates,
\begin{equation}
    u_\ell^{(j)} = u_\ell^{(j-1)}
    + \omega_j D_\ell^{-1}\bigl(b_\ell-A_\ell u_\ell^{(j-1)}\bigr),
    \qquad j=1,2,3.
    \label{eq:chebyshev_richardson}
\end{equation}

% Identify the resulting error polynomial and its smoothing bound.
For the initial error $e^{(0)}$, the three updates give
\begin{equation}
    e^{(3)}=p_3(D_\ell^{-1}A_\ell)e^{(0)},
    \qquad
    p_3(\lambda)
    =\prod_{j=1}^{3}(1-\omega_j\lambda)
    =\frac{T_3((c-\lambda)/d)}{T_3(c/d)},
    \label{eq:chebyshev_error_polynomial}
\end{equation}
where $T_3$ is the third Chebyshev polynomial. Hence
$\max_{\lambda\in[a,b]}|p_3(\lambda)|=|T_3(c/d)|^{-1}$. By comparison, three block-Jacobi steps using the same
optimal constant weight $2/(a+b)$ have the worst-case factor $(d/c)^3$. The varying weights therefore improve the
worst-case damping without changing the number of operator or block-inverse applications.

% Explain why the factorized realization has no additional vector traffic.
The factorized form in~\eqref{eq:chebyshev_richardson} deliberately avoids the auxiliary search-direction vector of a
three-term Chebyshev recurrence. It performs the same three operator applications, three block inversions, and three
solution--residual updates as three ordinary block-Jacobi Richardson steps; only the scalar weight changes between
stages. Thus Chebyshev acceleration introduces no additional full-vector stream. The three roots may be applied in an
order that limits intermediate residual growth; their order does not change the final polynomial in exact arithmetic.

\subsection{Treatment of the Refinement Edge}
\label{sec:mg_shadow_edge}

% Paragraph: the difficulty local smoothing faces at the refinement edge.
The cells adjacent to $E_\ell$ require coarse-side traces from outside the level-$\ell$ space. Classical local
smoothing treats interpolated coarse values as Dirichlet data; assembled implementations then split the level matrix
into interior and edge blocks and use dedicated edge operators for residual transfer~\cite{JanssenKanschat11}. Shadow
cells provide these data within the DG representation.

% Paragraph: the coarse side is frozen during level smoothing.
During smoothing on level $\ell$, the coarse-side data of the refinement edge do not change: the level-$\ell$ smoother
updates level-$\ell$ cells only, so the traces entering $A_\ell$ from across $E_\ell$ are constant throughout a
smoothing block, and the shadows carry them into the face loop of \secref{sec:faces} without masking or operator
splitting. Because the cycle runs in correction form, the value they are frozen at is not the coarse-side iterate but
the coarse-side \emph{correction} known so far, and it changes once per level visit. A level is entered with $u_\ell =
    0$, hence with zero shadows. This is the correct frozen value because no coarse correction has been computed yet, and
pre-smoothing runs against homogeneous inter-level Dirichlet data. Level prolongation then transfers the coarse-grid
correction onto level $\ell$, filling the shadows with the coarse-side share of that correction, and post-smoothing
runs against those values. This needs no dedicated step. The shadows are ordinary cells of the level, so
line~\ref{ln:level_prolong} of Algorithm~\ref{alg:dg_local_vcycle} fills them along with every other cell of the level.

% Paragraph: residual transfer through the shadows.
When the level residual $r_\ell=b_\ell-A_\ell u_\ell$ is evaluated after smoothing, the fine-shadow face kernels
accumulate flux contributions \emph{onto the shadow children}. These represent the couplings from the refined region
onto coarse-side test functions. Shadow assembly transfers them by the transposed embedding to the parent, where they
enter the coarser right-hand side with the restricted level residual. By Theorem~\ref{thm:shadow_exact}, these are the
non-matching face contributions of the active-mesh operator. The refinement edge therefore uses the shadow-embedding,
face-exchange, and shadow-assembly kernels.

% Paragraph: shadows are a property of the edge, shared by all level operators.
The shadow layer belongs to the refinement edge rather than to any one operator, and the active-mesh operator and all
level operators share it. What its blocks hold depends on the caller: for the Krylov operator they carry the prolonged
active solution, for a level operator the frozen coarse-side correction of that level visit. The contributions
accumulated on them undergo shadow assembly only in multigrid, fused with the level restriction; the Krylov loop leaves
them split, as \secref{sec:dg_pairing} permits.

%% NOTE: numerical results are in splitted/dg_results.tex, included by main_dg_shadow.tex.

%%%%%%%%%%%%%%%%%%%%%%%%
% Numerical results for the Laplace operator.
%%%%%%%%%%%%%%%%%%%%%%%%
%%%%%%%%%%%%%%%%%%%%%%%%
% Numerical results
%%%%%%%%%%%%%%%%%%%%%%%%
\section{Numerical Results}
\label{sec:results}

% State the questions addressed by the experiments.
The experiments address three questions. First, we examine how the penalty parameter and local refinement affect the
convergence of the multigrid-preconditioned solver. Second, we determine whether the volume, face, transfer, and
smoothing kernels operate near the hardware limits. Third, we measure how the complete operator and solver respond when
the number of local refinement levels increases. All experiments use polynomial degree $p=3$.

\subsection{Experimental Settings}
\label{sec:dg_results_settings}

% Define the hardware and solver settings.
All measurements are performed in FP64 on an NVIDIA A100 80\,GB SXM GPU. The V-cycle applies one pre-smoothing and one
post-smoothing application of the degree-three Chebyshev-accelerated block-Jacobi method described in
\secref{sec:mg_smoother}; each application comprises three Richardson updates with varying weights for the spectral
interval $[\lambda_{\min},\lambda_{\max}]=[0.2,2]$. The coarsest level is a single cell, i.e.\ $(p+1)^d=64$ unknowns,
and its system is solved exactly by a dense LU factorization assembled once by probing the level operator with unit
vectors. CG terminates when the preconditioned residual norm $\langle\tilde{r}_k,z_k\rangle^{1/2}$, the pairing-based
measure required by \secref{sec:dg_pairing}, has decreased by a factor $10^{-10}$.

% Define all reported rates and the DoF count used to normalize them.
Let $N_{\mathrm{dof}}=(p+1)^d N_{\mathrm{act}}$ denote the number of degrees of freedom on the active DG cells; shadow
degrees of freedom are excluded because they represent coarse-cell data rather than additional unknowns. For a timed
operation with wall time $t$, its throughput is $N_{\mathrm{dof}}/t$. We report the throughput of one complete SIPG
operator application $Au$, one block-Jacobi update, and the complete CG solve. The solve throughput is
$N_{\mathrm{dof}}/t_{\mathrm{solve}}$ and must therefore be read together with the iteration count.

% Define a refinement sequence that isolates the effect of adaptive depth.
We solve the Poisson problem on the Cartesian unit cube $\Omega=(0,1)^d$ with the manufactured solution $u(x)=\sin(\pi
    x_1)\sin(\pi x_2)\sin(\pi x_3)$, whose homogeneous Dirichlet data match the boundary condition of~\eqref{eq:sipg}. Each
mesh is identified by the number $n_{\mathrm{uni}}$ of uniform refinements followed by the number $n_{\mathrm{adap}}$
of nested local refinements. In each adaptive step, cells intersecting the radial shell $0.2<\lVert x-x_c\rVert_2<0.5$,
where $x_c=(1/2,\ldots,1/2)$ is the center of the domain, are marked for refinement. The case $n_{\mathrm{adap}}=0$ is
the uniformly refined reference mesh. For each adaptive depth, $n_{\mathrm{uni}}$ is chosen so that the largest level
contains enough cells to saturate the GPU while the total allocation remains below device capacity. The resulting mesh
is 2:1 balanced, as assumed in \secref{sec:shadow}.

% Show representative slices of the adaptive refinement sequence.
Figure~\ref{fig:dg_adaptive_meshes} shows planar slices through three meshes in the adaptive sequence. Identical view
limits and coloring distinguish the coarse and locally refined active cells, while dashed outlines show the shadow
layer generated along each refinement edge.

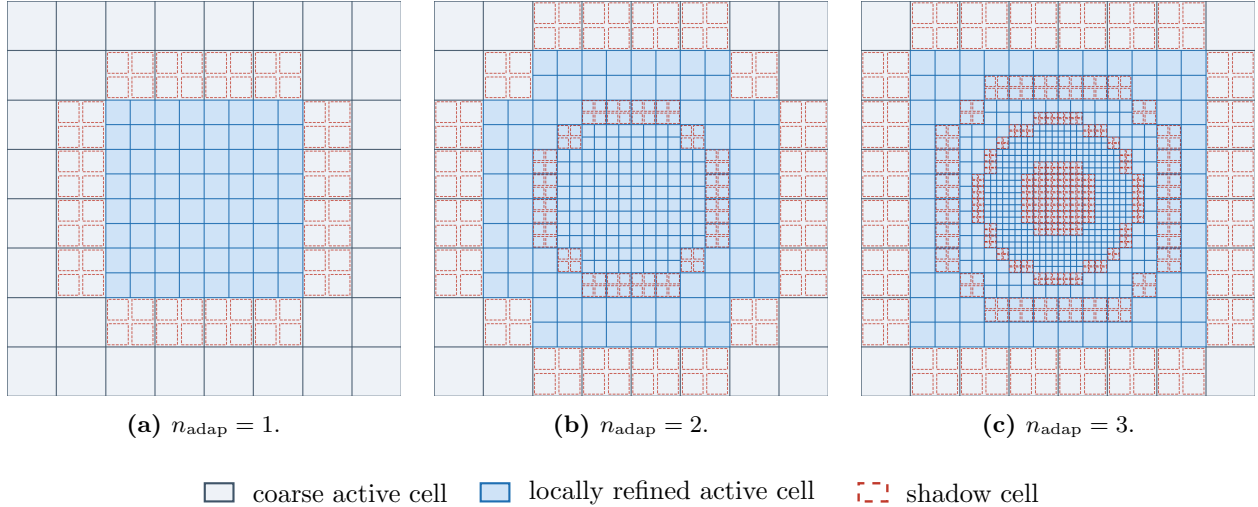
\begin{figure}[htbp]
    \centering
    \begin{subfigure}[b]{0.315\linewidth}
        \centering
        \includegraphics[width=\linewidth]{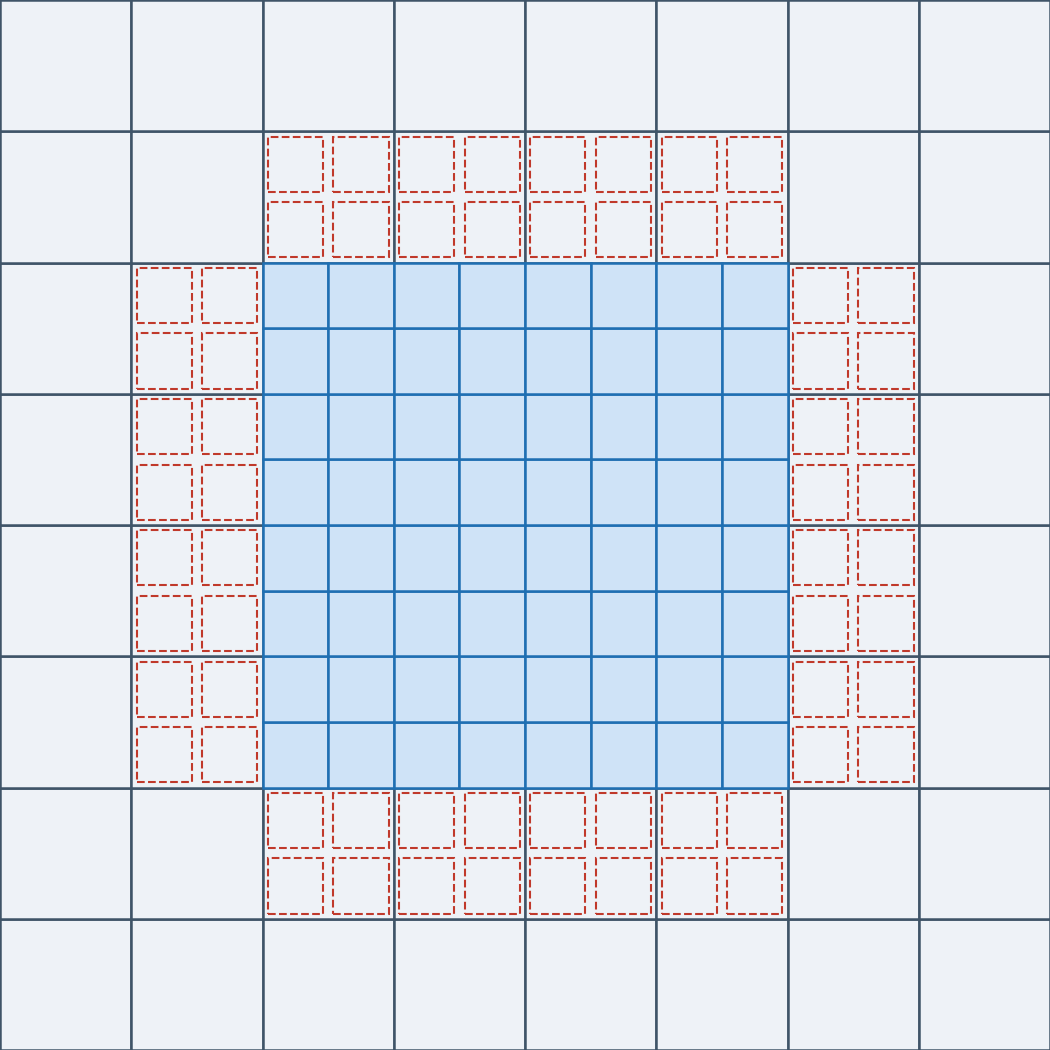}
        \caption{$n_{\mathrm{adap}}=1$.}
        \label{fig:dg_mesh_local_1}
    \end{subfigure}
    \hfill
    \begin{subfigure}[b]{0.315\linewidth}
        \centering
        \includegraphics[width=\linewidth]{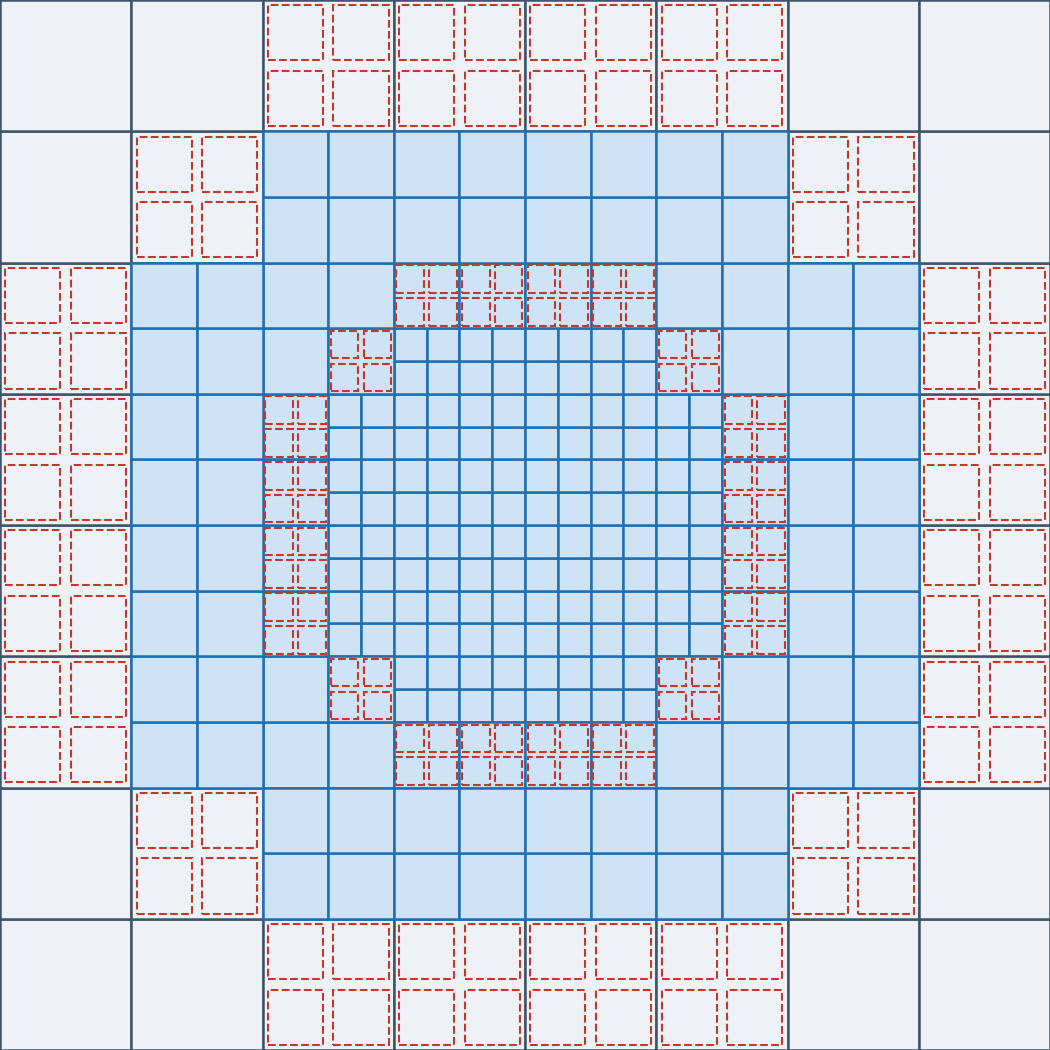}
        \caption{$n_{\mathrm{adap}}=2$.}
        \label{fig:dg_mesh_local_2}
    \end{subfigure}
    \hfill
    \begin{subfigure}[b]{0.315\linewidth}
        \centering
        \includegraphics[width=\linewidth]{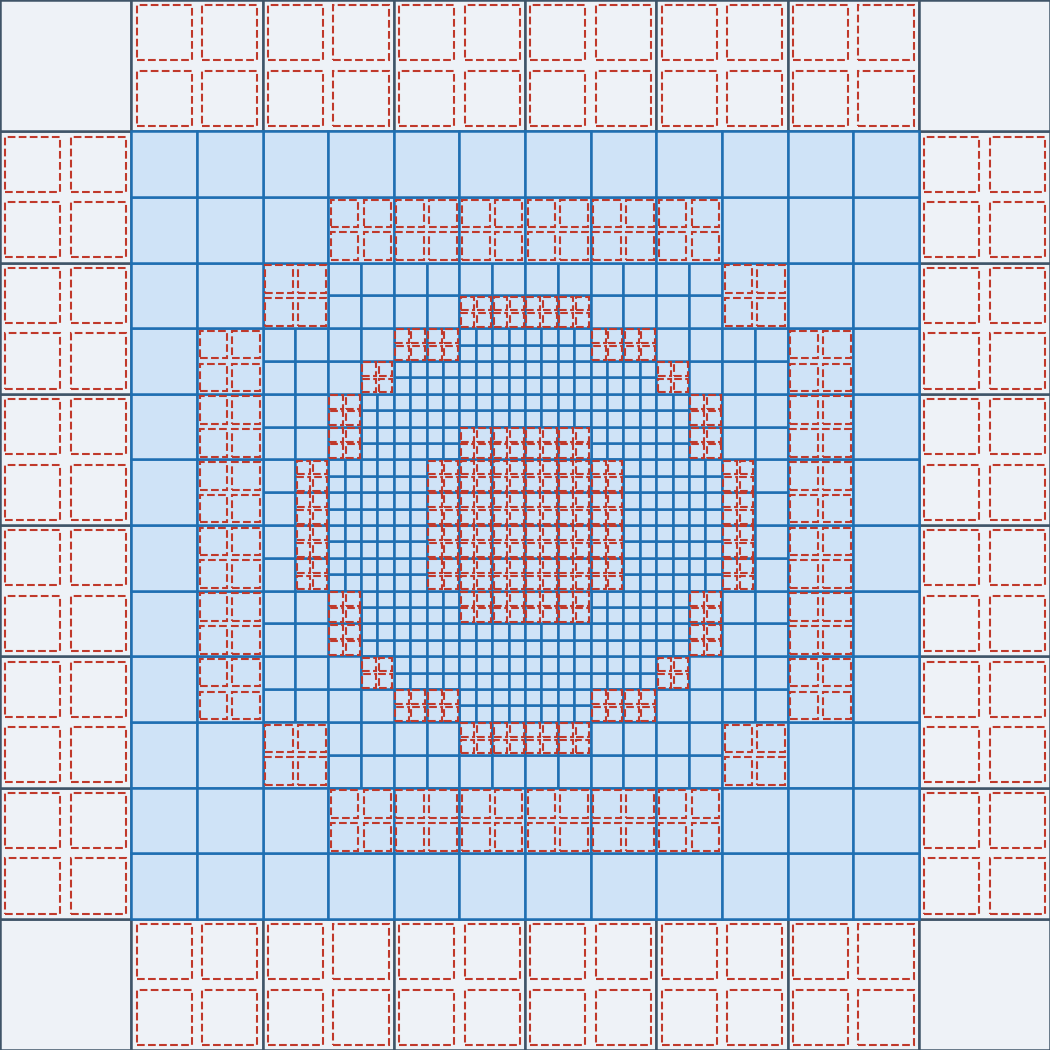}
        \caption{$n_{\mathrm{adap}}=3$.}
        \label{fig:dg_mesh_local_3}
    \end{subfigure}
    \par\medskip
    % Shared legend for the adaptive-mesh slice images.
\definecolor{dgcoarsefill}{RGB}{238,242,247}
\definecolor{dgcoarseedge}{RGB}{63,84,104}
\definecolor{dgrefinedfill}{RGB}{207,227,247}
\definecolor{dgrefinededge}{RGB}{31,111,178}
\definecolor{dgshadowedge}{RGB}{192,57,43}
\begin{tikzpicture}[baseline=(current bounding box.center), font=\small]
    \filldraw[fill=dgcoarsefill, draw=dgcoarseedge, line width=0.8pt]
        (0,0) rectangle (0.38,0.24);
    \node[anchor=west] at (0.50,0.12) {coarse active cell};

    \filldraw[fill=dgrefinedfill, draw=dgrefinededge, line width=0.8pt]
        (3.65,0) rectangle (4.03,0.24);
    \node[anchor=west] at (4.15,0.12) {locally refined active cell};

    \draw[dgshadowedge, dashed, line width=0.9pt]
        (8.65,0) rectangle (9.03,0.24);
    \node[anchor=west] at (9.15,0.12) {shadow cell};
\end{tikzpicture}
    \caption{Slices at $z=0.3$ through the locally refined Cartesian meshes used in the adaptive-depth experiment. All
        meshes have three global refinements ($n_{\mathrm{uni}}=3$); the local refinement depth is given below each
        panel. Solid rectangles are active cells, and dashed red rectangles are shadow cells. All panels use the same
        view limits and colors.}
    \label{fig:dg_adaptive_meshes}
\end{figure}

\subsection{Penalty Parameter and Solver Convergence}
\label{sec:dg_results_verification}

% Note the penalty a refinement edge inherits from the fine level.
The penalty parameter is $\sigma_F=C\,p(p+1)/h_F$ with $C=2$, i.e.\ $\sigma_F=24/h_F$ at $p=3$; the value $C=1$ is
close to the coercivity threshold of the form, where the coarse-level system becomes numerically singular. Note that on
a fine--shadow face both sides are cells of the fine level, so the face is penalized at the fine cell size, $\sigma_F =
    C\,p(p+1)/h_{\mathrm{fine}}$. The penalty of a non-matching face is thus dictated by the fine trace inverse inequality,
which means that the faces along a refinement edge carry twice the penalty of the coarse-side faces adjacent to them,
and the coupling across the edge is correspondingly stiffer than elsewhere on the coarse level. This stiffening is a
property of the refinement edge itself, not of how deeply the mesh is refined.

% Report convergence across adaptive depths.
CG converges in between $10$ and $12$ iterations over the complete refinement sequence, corresponding to an average
reduction of $\langle\tilde{r}_k,z_k\rangle^{1/2}$ per iteration of between $0.10$ and $0.15$. The count rises from
$10$ to $12$ at the first local refinement and is then independent of the adaptive depth
(Table~\ref{tab:dg_adaptive_throughput}): introducing a refinement edge costs two iterations, and deepening the
refinement costs none. We suspect that these two iterations are due to the doubled penalty the edge inherits from the
fine level, as discussed in the paragraph above.

\subsection{The Face Kernel}
\label{sec:dg_results_face_kernel}
%%%% NOTE: the benchmark calls the selected kernel `narrow`; the paper calls it `coalesced` to emphasize its relevant distinction from the other reported variants.

% State the memory operations of a face pass, which define the arithmetic-free references.
We compare two orderings of the face loop and then tune the storage block size $B$ for the one that is selected. For
one interior face, the kernel reads the value and normal-derivative layers from both adjacent cells, four
$(p+1)\times(p+1)$ tangential slabs, and accumulates into the same four layers of the output, which adds one read and
one write per entry. These four reads and four read--modify--writes are the memory operations that the arithmetic-free
references below retain. The bandwidth in Table~\ref{tab:dg_face_variants} is the DRAM throughput measured by NVIDIA
Nsight Compute.

% Exploit the rank-two structure of the normal flux coupling.
Both orderings exploit the algebraic rank of the normal coupling. With the trace data of a face ordered as
$z=(\partial_n u^-,u^-,u^+,\partial_n u^+)^T$, the jump and the average are extracted by
\begin{equation}
    j=(0,1,-1,0)^T,
    \qquad
    m=\left(\tfrac12,0,0,\tfrac12\right)^T.
\end{equation}
Apart from level-dependent geometric scaling, the SIPG flux matrix has the form
\begin{equation}
    F=-m j^T-j m^T+\sigma j j^T,
    \qquad \operatorname{range}(F)\subseteq\operatorname{span}\{j,m\}.
    \label{eq:dg_face_rank_two}
\end{equation}
Forming the jump $j^Tz$ and the average $m^Tz$ first therefore applies the tangential operator twice rather than once
per output row, four one-dimensional contractions per face instead of eight. The two kernels compared below share this factorization and differ in the order in
which they traverse the face loop. The face-major kernel uses the natural order: faces are numbered as the mesh
enumerates them and one thread block processes one face, so the face index is the slowest-varying dimension of the
data each block touches and the tangential indices are the fastest. It moves the adjacent value and derivative layers
of each cell as one depth-two tile, which is the fastest face-major arrangement we measured.

% Explain why the arithmetic-free memory-access reference can still be slow.
The \texttt{stream} measurement is an arithmetic-free reference for the original face-major access pattern. It performs
no contractions and does not evaluate the operator. In the blocked field layout, the cell lane is the fastest-varying
address, whereas the face index is the slowest axis of the reference tile; consequently, the threads of a warp span the
tangential indices of a single cell and are $8B$ bytes apart in memory. Its throughput is an upper bound imposed by
this addressing pattern alone, not by the flux arithmetic.

% Describe the coalesced ordering and tensor-core contractions.
The \texttt{coalesced} kernel instead sorts faces by the pair of adjacent memory blocks and then by the lane within the
block. It places the face index in the fastest dimension of a $(p+1)\times(p+1)\times B_F$ tile, so consecutive threads
access consecutive cell lanes. This is the one direction in which the blocked layout of \secref{sec:dg_cellwise} is
contiguous: for a fixed trace entry, the $B$ cells of a tile hold that entry in $8B$ consecutive bytes. A warp
therefore reads a single trace entry across consecutive lanes, and its $32$ requests fall into a few fully consumed
memory sectors, whereas the face-major ordering spreads the same $32$ requests over as many sectors of which one word
each is used. The two tangential contractions are grouped across the face batch into dense matrix multiplications and
executed on the FP64 tensor cores. Padding is applied only to the matrix dimension required by the hardware, while the
trace data retain their natural $4\times4$ extent. The transposed intermediate is written by swapping the two
tangential strides, avoiding an explicit permutation. The selected launch groups $B_F=64$ faces and uses two warps. An
arithmetic-free \texttt{stream} reference retains this ordering and the same memory operations but omits the
contractions, measuring the throughput ceiling of the selected access pattern.

% Quantify what the atomic-free sweep structure buys.
A face-centric loop writes both adjacent cells and therefore ordinarily requires atomics or face coloring. The
dimension-by-dimension sweep of \secref{sec:faces} needs neither: within the sweep for one orientation, each output
trace layer is written by exactly one interior face, and the three sweeps execute successively. The \texttt{atomic}
probe measures what this design decision is worth. It replaces the ordinary read--modify--write updates of the
\texttt{coalesced} kernel by FP64 atomic additions while leaving its contractions and addressing unchanged. On the
profiled $64^3$ grid, the operator throughput decreases from $17.44$ to $9.97$\,GDoF/s and the measured DRAM bandwidth
from $1230.1$ to $698.5$\,GB/s. Atomic accumulation would thus discard $42.8$\% of the attainable throughput, so
avoiding it is not a detail of the implementation but one of the larger single effects reported here.

% Compare the face-major and coalesced kernels with their memory references.
The face-major kernel reaches $5.60$\,GDoF/s on the larger grid. The face-major evaluation retains $93.0$\% of its
\texttt{stream} reference, showing that arithmetic adds little cost once the four layers are moved efficiently within
each face. The \texttt{coalesced} kernel reaches $22.67$\,GDoF/s and $1650.8$\,GB/s, corresponding to $81$\% of the
peak memory bandwidth and exceeding the face-major \texttt{stream} reference by a factor $3.77$. This comparison shows
that the memory-access order is more consequential than removing the face arithmetic. On the larger grid, its
\texttt{stream} reference reaches $23.65$\,GDoF/s, so the complete evaluation retains $95.9$\% of the throughput of its
arithmetic-free reference. All operator variants agree with the independent reference to a maximum relative difference
of $1.79\times10^{-16}$. The \texttt{coalesced} kernel is used in all remaining experiments.

\begin{table}[htbp]
    \centering
    % NVIDIA Nsight Compute measurements: raw_results/dg/interior_face.output (A100 80GB).
\begin{adjustbox}{max width=\linewidth}
    \begin{tabular}{@{}llccccc@{}}
        \toprule
        {Cells} & {Metric}
                & \multicolumn{2}{c}{\textbf{Face-major}}
                & \multicolumn{3}{c}{\textbf{Coalesced}}                                                                          \\
        \cmidrule(lr){3-4}\cmidrule(lr){5-7}
                & & {Evaluation} & {\texttt{stream}} & {Evaluation} & {\texttt{stream}} & {\texttt{atomic}} \\
        \midrule
        \multirow{2}{*}{$64^3$}
                & {Throughput [GDoF/s]}
                & {5.77}                                  & {6.11}       & {\textbf{23.14}}  & {24.71}      & {9.97}              \\
                & {DRAM bandwidth [GB/s]}
                & {413.8}                                 & {435.9}      & {\textbf{1623.4}} & {1722.6}     & {698.5}             \\
        \midrule
        \multirow{2}{*}{$128^3$}
                & {Throughput [GDoF/s]}
                & {5.60}                                  & {6.02}       & {\textbf{22.67}}  & {23.65}      & {10.06}             \\
                & {DRAM bandwidth [GB/s]}
                & {414.1}                                 & {442.6}      & {\textbf{1650.8}} & { 1736.8 }   & {731.6}             \\
        \bottomrule
    \end{tabular}
\end{adjustbox}

    \caption{Interior-face kernel variants and memory-access references on uniformly refined Cartesian cubes with
        $64^3$ and $128^3$ cells ($16\,777\,216$ and $134\,217\,728$ DoFs; A100 80\,GB SXM,
        FP64). Each evaluation is paired with an arithmetic-free reference using the same memory-access
        pattern. The two evaluations differ only in the order in which they
        traverse the face loop. The references omit the contractions and do not evaluate the operator;
        the \texttt{atomic} column is an exact \texttt{coalesced} evaluation using FP64 atomic additions. GDoF/s is
        derived from measured kernel time, and the bandwidth is the DRAM throughput reported by NVIDIA Nsight Compute.
        The \texttt{coalesced} evaluation is used by the DG operator.}
    \label{tab:dg_face_variants}
\end{table}

% Tune the block size for the selected ordering; measurements: raw_results/sweep_p3.csv and raw_results/dg/rsweep_face.output, lines 1--105.
With the ordering fixed, the storage block size $B$ remains, and it is shared with the volume kernel.
Figure~\ref{fig:dg_layout_throughput} measures its effect on both degree-three kernels. The volume kernel is
comparatively insensitive to the layout: its throughput varies from $86.48$ to $96.12$\,GDoF/s over the measured block
sizes. In contrast, the face rate increases from $14.98$\,GDoF/s at $B=1$ to a maximum of $22.38$\,GDoF/s at $B=32$, an
improvement of $49.4$\%. At $B=32$, the volume kernel retains $96.9$\% of its maximum throughput. Thus, the blocked
layout materially improves face evaluation without sacrificing the volume evaluation. The sweep uses the $128^3$ grid
of Table~\ref{tab:dg_face_variants}, and its maximum agrees with the $22.67$\,GDoF/s reported there to within one
percent.

\begin{figure}[htbp]
    \centering
    \begin{minipage}{0.49\linewidth}
        \centering
        \makebox[\linewidth][c]{% Block-size sweep for the degree-three volume and coalesced interior-face kernels.
% Volume source: raw_results/sweep_p3.csv, columns block_data and vol_gdofs.
% Face source: raw_results/dg/rsweep_face.output, lines 1--105; benchmark variant `narrow`.
\begin{tikzpicture}
    \begin{loglogaxis}[
            log basis x=2,
            log basis y=2,
            width=\linewidth,
            height=0.75\linewidth,
            xlabel={Block size $B$},
            ylabel={Throughput [GDoF/s]},
            xmin=0.8, xmax=160,
            ymin=8, ymax=128,
            xtick={1,2,4,8,16,32,64,128},
            xticklabels={1,2,4,8,16,32,64,128},
            ytick={8,16,32,64,128},
            yticklabels={8,16,32,64,128},
            grid=both,
            legend style={font=\scriptsize, at={(0.97,0.7)}, anchor=north east},
            legend cell align=left,
        ]
        % Volume throughput; column index 3 is vol_gdofs.
        \addplot[orange, mark=triangle*]
        table[col sep=comma, x index=0, y index=3] {raw_results/sweep_p3.csv};

        % Coalesced face throughput; `narrow` in the benchmark output.
        \addplot[blue, mark=square*, dashed] table {
                B   throughput
                1   14.98
                2   15.88
                4   20.52
                8   20.43
                16  22.08
                32  22.38
                64  22.21
                128 21.91
            };
        \legend{Volume, Face evaluation}
    \end{loglogaxis}
\end{tikzpicture}}
    \end{minipage}
    \caption{Effect of the storage block size $B$ on the degree-three Cartesian volume and \texttt{coalesced} interior-face
        kernels on the A100 80\,GB SXM in FP64. Solid and dashed lines denote the volume and face kernels,
        respectively.}
    \label{fig:dg_layout_throughput}
\end{figure}
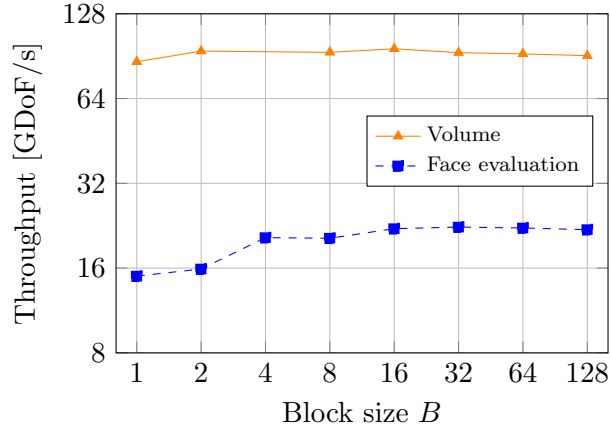

\subsection{Roofline Analysis}
\label{sec:dg_results_roofline}

% Explain the scope and normalization of the roofline experiment.
Figure~\ref{fig:dg_roofline} places the kernels used by one V-cycle on the A100 roofline. The volume point is the same
seven-contraction Cartesian Laplace kernel as in~\eqref{eq:eval_pipeline}. The level-prolongation and level-restriction
points are the tensor-product transfers used by the multigrid hierarchy. The remaining points measure the block-Jacobi
smoother and the face kernel.

% State the tensor-core instruction the dense contractions rely on.
The dense contractions of the volume kernel, the smoother, and the batched face contractions are issued as FP64
matrix-multiply-accumulate (MMA) instructions of shape $8\times8\times4$, the only FP64 tensor-core shape the A100
provides. The dashed ceiling of Figure~\ref{fig:dg_roofline} is the peak rate of that instruction. Its fixed $8\times8$
operand shape is also what forces the padding noted below: at $p=3$ the natural tile extent is $(p+1)=4$, half of the
instruction's eight-wide operand.

% Explain the hardware-counter operation counts and tensor-core padding.
The plotted FLOP rates and arithmetic intensities are hardware-counter measurements collected with NVIDIA Nsight
Compute through \texttt{ncu}, so the operation counts include arithmetic that the operator does not use. Two effects
account for this. In the smoother, the $(p+1)=4$ operand fills only half of the $8\times8\times4$ tile, so half of the
issued MMA arithmetic is useful. In the volume kernel, the split cascade issues ten one-dimensional contractions but
consumes seven of them, so $70$\% of its arithmetic is useful. The FLOP rates are therefore upper bounds, while the
measured bandwidth is not affected.

\begin{figure}[htbp]
    \centering
    % Roofline inputs and provenance for the p=3 DG experiment.
% Volume: p=3 row of raw_results/laplace_roofline.csv.
% Smoother: NVIDIA Nsight Compute measurement of fdm_triton.
% Transfers: p=3 entries of the measured transfer series in fig_mg_roofline.tex.
% Face: narrow_r6 row of raw_results/dg/interior_face.output.
\def\DGFaceAI{1.40}
\def\DGFaceTF{2.200}

\begin{tikzpicture}
    \begin{loglogaxis}[
            width=0.62\linewidth,
            height=0.43\linewidth,
            xlabel={Arithmetic intensity [FLOP/byte]},
            ylabel={Performance [TFLOP/s]},
            log basis x=2,
            xmin=0.25, xmax=16,
            xtick={0.25,0.5,1,2,4,8,16},
            xticklabels={$\frac{1}{4}$,$\frac{1}{2}$,1,2,4,8,16},
            log basis y=2,
            ymin=0.25, ymax=32,
            ytick={0.25,0.5,1,2,4,8,16,32},
            yticklabels={$\frac{1}{4}$,$\frac{1}{2}$,1,2,4,8,16,32},
            grid=both,
            legend pos=south east,
            legend columns=1,
            legend cell align=left,
            legend style={font=\small},
        ]
        % Memory-bound and vector-compute regions.
        \addplot[green, fill=green, draw=none, fill opacity=0.3, forget plot]
        coordinates {(0.001,0.002039) (4.76,9.7) (1000,9.7) (1000,0.0001) (0.001,0.0001)};
        \addplot[cyan, fill=cyan, draw=none, fill opacity=0.3, forget plot]
        coordinates {(4.76,9.7) (9.56,19.5) (1000,19.5) (1000,9.7)};

        % A100 80GB SXM rooflines in fp64.
        \addplot[thick, domain=0.001:9.56, samples=2, forget plot] {2.039*x};
        \addplot[thick, domain=0.001:1000, samples=2, forget plot] {9.7};
        \addplot[thick, dashed, domain=0.001:1000, samples=2, forget plot] {19.5};

        % Hardware-ceiling annotations, matching the rooflines in the related papers.
        \node[rotate=29.3, above, font=\small\bfseries] at (axis cs:0.72,1.47)
        {Bandwidth 2039 GB/s};
        \node[above right, font=\small\bfseries] at (axis cs:0.26,9.7)
        {CUDA 9.7 TFLOP/s};
        \node[above right, font=\small\bfseries] at (axis cs:0.26,19.5)
        {Tensor 19.5 TFLOP/s};

        % Operator kernels: volume and interior faces use circular markers.
        \addplot[black, mark=*, mark size=3pt, only marks] coordinates {(5.000,8.140)};
        \addlegendentry{Volume}
        \addplot[violet, mark=o, mark options={fill=none, line width=1.2pt}, mark size=3pt, only marks]
        coordinates {(\DGFaceAI,\DGFaceTF)};
        \addlegendentry{Face}

        % Smoother kernel: the block-Jacobi inverse uses a star marker.
        \addplot[blue, mark=star, mark options={fill=none, line width=1.2pt}, mark size=3.5pt, only marks]
        coordinates {(4.300,6.481)};
        \addlegendentry{Blockwise-Jacobi}

        % Multigrid transfers: level restriction and level prolongation use square markers.
        \addplot[red, mark=square*, mark size=4pt, mark options={line width=1pt}, only marks]
        coordinates {(1.714,2.117)};
        \addlegendentry{Level restriction}
        \addplot[orange, mark=square, mark options={fill=none, line width=1.2pt}, mark size=4pt, only marks]
        coordinates {(0.888,1.445)};
        \addlegendentry{Level prolongation}
    \end{loglogaxis}
\end{tikzpicture}
    \caption{
        Roofline of the DG and multigrid building blocks on the A100 80\,GB SXM in FP64. Markers denote
        measurements of the Cartesian volume kernel, face kernel, block-Jacobi smoother, and geometry-independent
        level-transfer kernels. Solid lines are the memory-bandwidth slope (2039\,GB/s) and vector FP64 ceiling
        (9.7\,TFLOP/s); the dashed line is the tensor-core FP64 ceiling (19.5\,TFLOP/s).}
    \label{fig:dg_roofline}
\end{figure}
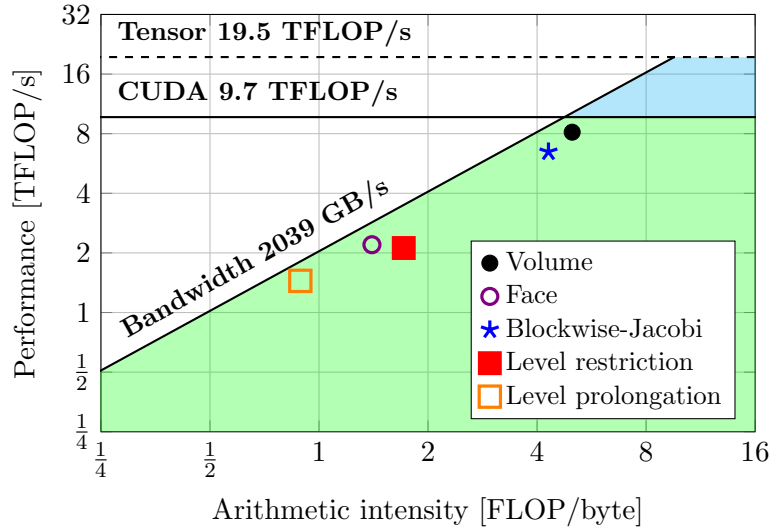

% Interpret the measured kernel positions on the roofline.
The volume kernel reaches $8.14$\,TFLOP/s at an arithmetic intensity of $5.00$\,FLOP/byte. Level restriction and level
prolongation reach $2.12$ and $1.45$\,TFLOP/s, respectively, while the smoother and face kernels reach $6.48$ and
$2.20$\,TFLOP/s. The face kernel sits in the memory-bound region. As the grid size increases and the device approaches
saturation, its measured DRAM bandwidth reaches $1650.8$\,GB/s on the finest mesh, or $81$\% of the peak memory
bandwidth. Neither compute ceiling is the relevant limit for this kernel, and for the same reason the unused arithmetic
above costs no performance: the kernels are bound by device memory, not by the arithmetic they issue.

\subsection{Operator and Solve Throughput}
\label{sec:dg_results_throughput}

% Introduce and interpret the global-refinement experiment; measurements: raw_results/dg/thougputs.output, lines 5--113.
We first measure globally refined meshes to establish the problem sizes required to saturate the GPU and to provide a
uniform-mesh reference for the local-refinement experiment. Table~\ref{tab:dg_global_throughput} reports the complete
$Au$ application, block-Jacobi update, and CG solve. The kernel throughputs increase rapidly up to $n_{\mathrm{uni}}=6$
and improve more moderately on the largest mesh: from $n_{\mathrm{uni}}=6$ to 7, the $Au$ and block-Jacobi rates
increase from $15.55$ to $17.36$\,GDoF/s and from $9.90$ to $10.56$\,GDoF/s, respectively. Across the global-refinement
sequence, CG requires 10--11 iterations and the solve throughput rises from $3.3$ to $89.6$\,MDoF/s. For reference, the
multilevel interior penalty solver of~\cite{cui2025multilevel} reports $61.86$\,MDoF/s at $p=3$ on the same device, for
a relative tolerance of $10^{-8}$ against the $10^{-10}$ used here. That solver uses a vertex-patch smoother, which
converges in fewer iterations than the cell-wise block-Jacobi smoother of \secref{sec:mg_smoother}; an optimized patch
smoother would therefore be expected to outperform the present configuration. Nothing in the shadow construction
prevents one, and matrix-free local solvers for such patches are available~\cite{wichrowski2025local}: the smoother
only has to update the level-$\ell$ unknowns and leave the shadows frozen (\secref{sec:mg_smoother}), and a patch
smoother on the blocked layout satisfies this as readily as the cell-wise one.

\begin{table}[htbp]
    \centering
    % Source: raw_results/dg/thougputs.output, lines 5--15, 38--47, 72--80, and 106--113.
% Throughput under global refinement; N_dof counts active DG unknowns only.
\begin{tabular}{@{}c|rr|rrr@{}}
    \toprule
    {$n_{\mathrm{uni}}$} & {Cells} & {CG it.} & {$Au$} & {Smoother} & {Solve} \\
                         &         &          & \multicolumn{2}{c}{[GDoF/s]} & {[MDoF/s]} \\
    \cmidrule(lr){4-5}
    4 & 4,096     & 11 & 1.63  & 1.27  & 3.3 \\
    5 & 32,768    & 10 & 10.61 & 7.67  & 22.4 \\
    6 & 262,144   & 10 & 15.55 & 9.90  & 74.4 \\
    7 & 2,097,152 & 10 & 17.36 & 10.56 & 89.6 \\
    \bottomrule
\end{tabular}

    \caption{Throughput of the SIPG operator, block-Jacobi update, and multigrid-preconditioned CG solver under global
        refinement (A100 80\,GB SXM, FP64). The operator and smoother rates are reported in GDoF/s, while the
        complete-solve rate is reported in MDoF/s.}
    \label{tab:dg_global_throughput}
\end{table}

% Introduce the local-refinement experiment; measurements: raw_results/dg/thougputs.output, lines 38--47 and 140--245.
We next fix $n_{\mathrm{uni}}=6$ and increase only the local refinement depth. The $n_{\mathrm{adap}}=0$ row in
Table~\ref{tab:dg_adaptive_throughput} is therefore the corresponding uniform-mesh result from
Table~\ref{tab:dg_global_throughput}.

\begin{table}[htbp]
    \centering
    % Source: raw_results/dg/thougputs.output, lines 38--47, 140--150, 170--181, 199--211, and 232--245.
% Throughput under local refinement; N_dof counts active DG unknowns only.
\begin{tabular}{@{}cc|rr|rrr@{}}
    \toprule
    {$n_{\mathrm{uni}}$} & {$n_{\mathrm{adap}}$} & {Cells} & {CG it.} & {$Au$} & {Smoother} & {Solve} \\
                         &                       &         &          & \multicolumn{2}{c}{[GDoF/s]} & {[MDoF/s]} \\
    \cmidrule(lr){5-6}
    6 & 0 & 262,144   & 10 & 15.55 & 9.90 & 74.4 \\
    6 & 1 & 295,632   & 12 & 13.66 & 5.11 & 57.5 \\
    6 & 2 & 465,256   & 12 & 12.48 & 5.32 & 48.4 \\
    6 & 3 & 1,158,200 & 12 & 11.53 & 5.43 & 42.2 \\
    6 & 4 & 3,936,304 & 12 & 11.12 & 5.44 & 39.5 \\
    \bottomrule
\end{tabular}

    \caption{Throughput of the SIPG operator, block-Jacobi update, and local multigrid-preconditioned CG solver under
        local refinement with $n_{\mathrm{uni}}=6$ (A100 80\,GB SXM, FP64). The operator and smoother rates
        are reported in GDoF/s, while the complete-solve rate is reported in MDoF/s. All rates are normalized by active
        DoFs; shadow DoFs are excluded from $N_{\mathrm{dof}}$. The $n_{\mathrm{adap}}=0$ row is the uniform reference
        mesh.}
    \label{tab:dg_adaptive_throughput}
\end{table}

% Compare local-refinement throughput; measurements: raw_results/dg/thougputs.output, lines 38--47 and 140--245.
Across problems containing 262,144 to 3,936,304 active cells, the complete operator sustains $11.12$--$15.55$\,GDoF/s
and the block-Jacobi update sustains $5.11$--$9.90$\,GDoF/s. At four local refinement levels, these rates are $71.5$\%
and $54.9$\% of their uniform-mesh values, respectively, while the solve throughput decreases from $74.4$ to
$39.5$\,MDoF/s. The iteration count increases from 10 to 12 after the first local refinement and remains constant
thereafter. Thus, the continued throughput reduction with increasing local depth is caused by refinement-edge work
rather than by a further loss of multigrid convergence.

%%%%%%%%%%%%%%%%%%%%%%%%
% Conclusion
%%%%%%%%%%%%%%%%%%%%%%%%
\section{Conclusion}
\label{sec:conclusion}

% Paragraph: summary of the shadow-cell method.
We introduced a matrix-free symmetric interior penalty discontinuous Galerkin method for adaptively refined Cartesian
meshes on GPUs. The method relies on auxiliary shadow cells to handle the topological irregularities of local
refinement. By projecting the coarse-cell polynomial onto virtual fine-level children, shadows eliminate non-matching
interfaces. This allows the solver to rely entirely on a uniform, dimensionally split face-evaluation kernel, avoiding
the irregular addressing and warp divergence associated with standard subface integration.

% Paragraph: theoretical properties and solver integration.
We proved that the shadow evaluation reproduces the face integrals and the quadrature of the standard non-matching
formulation exactly. The discrete problem, and hence its approximation properties, are therefore those of the standard
method. To mitigate the memory-bandwidth bottleneck of face integration, we combined a blocked storage layout with a
Hermite-type basis, limiting inter-cell data dependencies to the immediate face layers. Additionally, the primal--dual
formulation integrates the redundant shadow storage into the Krylov loop, eliminating shadow assembly during the
operator application. In the local geometric multigrid preconditioner, the shadows carry the inter-level boundary data
and transfer the coarse-grid residual corrections.

% Paragraph: performance and hardware utilization.
Numerical experiments on an NVIDIA A100 show that the uniform face exchange and blocked data layout reach $81$\% of
peak device memory bandwidth, within $4.1$\% of an arithmetic-free reference using the same access pattern. Tensor-core
utilization further accelerates the coalesced tangential contractions. The approach maintains efficiency under adaptive
refinement: at four local refinement levels the operator retains $71.5$\% of its uniform-mesh throughput, while the
local multigrid preconditioner keeps the iteration count constant. The shadow-cell method thus maps the geometric
irregularity of adaptive mesh refinement onto the concurrency requirements of GPU architectures. Its present scope is
the Cartesian background mesh, on which the kernels stream no geometry; complex domains would be reached by combining
it with an unfitted or shifted-boundary layer~\cite{burman2015cutfem,wichrowski2025matrix}, which we leave to future
work.

\paragraph{Declarations}
The author declares support of two local feline agents (Micro and Conda) running locally alongside language models
(Gemini, Claude, ChatGPT) during text drafting. The final manuscript was audited by the author, who retains full
accountability for all scientific content.

\bibliographystyle{siam}
\bibliography{literature,added_literature}

\end{document}